\documentclass[reqno,12pt]{amsart}
\usepackage[margin=1.2in]{geometry}
\usepackage{amssymb}
\usepackage{amsthm}
\usepackage{amsmath}
\usepackage{amsxtra}
\usepackage{latexsym}
\usepackage{mathrsfs}
\usepackage[all,cmtip]{xy}
\usepackage[all]{xy}
\usepackage{enumitem}
\usepackage{xcolor}
\usepackage{comment}
\usepackage{mathabx,epsfig}

\usepackage[
backend=biber,
bibstyle=ieee-alphabetic,
citestyle=ieee-alphabetic,
sorting=nyt,
isbn=false,
url=false,
doi=false,
giveninits=true,
maxnames=10,
labelalpha=true,
maxalphanames=4,
dashed=false,
]{biblatex}
\DeclareCaseLangs{}

\DeclareFieldFormat*{date}{\mkbibparens{#1}}
\DeclareFieldFormat*{volume}{\mkbibbold{#1}}
\DeclareFieldFormat*{title}{\mkbibemph{#1\isdot}}
\DeclareFieldFormat*{journaltitle}{#1}
\DeclareFieldFormat*{pages}{{#1}}

\AtEveryBibitem{%
  \ifentrytype{online}{%
    \DeclareFieldFormat*{date}{#1}%
  }{%
  }%
}

\newtheorem{thm}{Theorem}[section]
\newtheorem{lem}[thm]{Lemma}

\newtheorem{prop}[thm]{Proposition}
\newtheorem{cor}[thm]{Corollary}

\newtheorem{assump}[thm]{Assumption}

\theoremstyle{definition}
\newtheorem{defn}[thm]{Definition}
\newtheorem{rmk}[thm]{Remark}
\newtheorem*{ack}{Acknowledgements}
\numberwithin{equation}{section}

\def\C{{\mathbb C}}
\def\Q{{\mathbb Q}}
\def\R{{\mathbb R}}
\def\Z{{\mathbb Z}}
\def\P{{\mathbb P}}
\def\A{{\mathbb A}}

\def\Hom{\mathop{\mathrm{Hom}}\nolimits}

\def\NE{\overline{NE}}
\def\sO{\mathcal{O}}

\DeclareMathOperator{\Pic}{Pic}
\DeclareMathOperator{\rank}{rank}

\DeclareMathOperator{\id}{id}

\DeclareMathOperator{\Spec}{Spec}
\DeclareMathOperator{\Exc}{Exc}
\DeclareMathOperator{\Supp}{Supp}

\DeclareMathOperator{\Tr}{Tr}

\DeclareMathOperator{\ord}{ord}
\DeclareMathOperator{\Diff}{Diff}
\DeclareMathOperator{\chara}{char}

\AtBeginDocument{%
  \def\MR#1{}
}

\title[]
{Minimal model program for semi-stable threefolds in mixed characteristic}
\author{Teppei Takamatsu, Shou Yoshikawa}

\address{Department of Mathematics, Graduate School of Science, Kyoto University, Kyoto 606-8502, Japan}
\email{teppeitakamatsu.math@gmail.com}

\address{RIKEN Interdisciplinary Theoretical and Mathematical Sciences (iTHEMS), Wako, Saitama 351-0198, Japan}
\email{shou.yoshikawa@riken.jp}

\begin{document}

\begin{abstract}
In this paper, we study the minimal model theory for threefolds in mixed characteristic.
As a generalization of a result of Kawamata, we show that the MMP holds for strictly semi-stable schemes over an excellent Dedekind scheme $V$ of relative dimension two without any assumption on the residue characteristics of $V$.
We also prove that we can run a  $(K_{X/V}+\Delta)$-MMP over $Z$, where $\pi \colon X \to Z$ is a projective birational morphism of $\Q$-factorial quasi-projective $V$-schemes  and  $(X,\Delta)$ is a three-dimensional dlt pair with $\Exc(\pi) \subset \lfloor \Delta \rfloor $.
\end{abstract}

\maketitle

\setcounter{tocdepth}{1}

\section{Introduction}
The \emph{Minimal model program} (MMP, for short), which is a higher-dimensional analog of the classification method of surfaces, is a tool to find a ``simplest'' variety in each birational equivalence class.
In characteristic zero, this program holds for threefolds and  varieties of general type (see \cite{bchm}).
The minimal model theory recently has been studied also in positive characteristic,
and the program is now known to hold for threefolds over a perfect field of characteristic $p>3$ (see \cite{hacon-xu}, \cite{ctx}, \cite{birkar16}, \cite{gnt}, \cite{birkar-waldron}, \cite{hacon-witaszek19b}).

The MMP is also studied for schemes not necessarily defined over a field.
%One of the motivations is to understand properties of reduction of varieties over an integer ring.
%One of the motivation of such a generalization is to study \emph{reduction} of a variety $X$ over a number field. 
Such a generalization of the MMP plays an important role to construct a nice model $\mathcal{W}$ over an integer ring whose generic fiber is isomorphic to a given variety $W$ (see \cite{Matsumoto15} \cite{liedtke-matsumoto}, \cite{CLL}, \cite{chiarellotto-lazda}).
The MMP is known to hold for excellent surfaces (\cite{tanaka18}) and strictly semi-stable schemes over an excellent Dedekind scheme of relative dimension two whose each residue characteristic is neither $2$ nor $3$ (\cite{kawamata94}, \cite{kawamata99}).
In this paper, we study the MMP for threefolds over an excellent Dedekind scheme.

Our first main result is a generalization of the result of Kawamata.
He used the classification of singularities, which depends on residue characteristic, to prove the existence of flips.
We use a completely different approach to prove the existence of flips without assumption on the residue characteristics.

\begin{thm}\label{thm:semi-stable mmp}\textup{(Theorem \ref{thm:semi-stable mmp'})}
Let $V$ be an excellent Dedekind scheme.
Let $X$ be a scheme which is strictly semi-stable over $V$ of relative dimension two.
Let $X \to Z$ be a projective morphism to a quasi-projective scheme $Z$ over $V$.
Then we can run a $K_{X/V}$-MMP over $Z$ 
%preserving Assumption \ref{assump:semi-stable}
which terminates with a minimal model or a Mori fiber space.
Furthermore, this program preserves good conditions (see Assumption \ref{assump:semi-stable}), for example, the output $Y$ of this MMP is Cohen-Macaulay and every irreducible component of each closed fiber of $Y \to V$ is geometrically normal.
\end{thm}

\noindent
%Since Kawamata's result is used in several studies of reductions of varieties over an integer ring, Theorem \ref{thm:semi-stable mmp} also has immediate applications to those studies, as follows.
Kawamata's result is used in several studies of reductions of varieties over an integer ring.
Therefore, we also generalize such studies to the case where the residue characteristic is $2$ or $3$, as the following.
\begin{description}
    \item[Good reduction criterion for K3 surfaces (\cite{Matsumoto15}, \cite{liedtke-matsumoto}, \cite{CLL})]
    \textup{ }\\
    Let $K$ be a complete discrete valuation field with perfect residue field of characteristic $p$ and $X$ a K3 surface over $K$. Suppose that $X$ admits potentially semi-stable reduction. 
    If the $G_{K}$-representation $H^{2}_{\textup{\'{e}t}}(X_{\bar{K}}, \Q_{\ell})$ is unramified for some $\ell \neq p$, then $X$ admits good reduction after an unramified extension of $K$.
    \item[Abelian surfaces have potentially combinatorial reduction (\cite{chiarellotto-lazda})]
    \textup{ }\\
    Let $K$ be a complete discrete valuation field with perfect residue field of characteristic $p$ and $X$ an abelian surface over $K$.
    Then $X$ admits potentially combinatorial reduction in the sense of \cite[Definition 10.1]{chiarellotto-lazda}.  %(\cite[Theorem 10.3]{chiarellotto-lazda}).
    Such a model give a compactification of a N\'{e}ron model, and the dual graph of the special fiber can be classified (see Theorem \ref{thm:combinatorial} and Proposition \ref{prop:compactneron}).
    %We note that the dual graphs of the special fibers of such models are classified in \cite[Section 8]{chiarellotto-lazda} under the assumption $p \neq 2$. 
\end{description}

\noindent
Our second main result is a generalization of \cite[Theorem 1.1]{hacon-witaszek19} to the mixed characteristic case.

\begin{thm}\label{thm:birational relative mmp}\textup{(Theorem \ref{thm:birational relative mmp'})}
Let $V$ be an excellent Dedekind scheme.
%Let $X$ be a $\Q$-factorial quasi-projective integral scheme over $V$ of dimension three.
Let $(X,\Delta)$ be a dlt pair, where $X$ is a $\Q$-factorial integral scheme which is flat and quasi-projective over $V$ of relative dimension two.
Assume that there exists a projective birational morphism $\pi \colon X \to Z$ to a normal $\Q$-factorial variety $Z$ with $\Exc(\pi) \subset \lfloor \Delta \rfloor$.
Then we can run a $(K_{X/V}+\Delta)$-MMP over $Z$ which terminates with a minimal model.
\end{thm}

\noindent
The existence of dlt modifications and the inversion of adjunction for klt pairs follow from Theorem \ref{thm:birational relative mmp} as a corollary (see Corollary \ref{cor:dlt modification}, \ref{cor:inversion of adjunction}).

One of the key ingredients of the proofs of Theorems \ref{thm:semi-stable mmp} and \ref{thm:birational relative mmp} is to prove the existence of pl-flips with ample divisor in the boundary.
Indeed, all flips appearing in the proof of Theorem \ref{thm:birational relative mmp} are of this type and the existence of necessary flips for Theorem \ref{thm:semi-stable mmp} is reduced by an argument in \cite{hacon-witaszek19} to the existence of flips of this type.
In positive characteristic, \cite[Theorem 1.3]{hacon-witaszek20} proved the existence of pl-flips with ample divisor in the boundary using global $F$-regularity and the vanishing theorem up to Frobenius twist.
We employ the same strategy in mixed characteristic, replacing  global $F$-regularity with \emph{global $T$-regularity} and the Frobenius morphism with alterations.
The global $T$-regularity of a log pair $(X,\Delta)$ over an excellent Dedekind scheme $V$ is defined by the surjectivity of the map
\[
H^0(\omega_{Y/V}(\lceil -\pi^*(K_{X/V}+\Delta) \rceil)) \to H^0(\sO_X)
\]
induced by the Grothendieck trace map for every alteration $\pi \colon Y \to X$.
%An analogous argument works in mixed characteristic, using projective generically finite morphisms, called \emph{alterations}, and the \emph{global $T$-regularity}, defined later, instead of them, respectively.
The vanishing theorem up to alterations is obtained as a corollary of \cite[Theorem 6.28]{bhatt20} (see Corollary \ref{cor:zero map}).
As a consequence, we obtain the following theorem which is an analog of \cite[Theorem 1.3]{hacon-witaszek20}.

\begin{thm}\label{thm:pl-filp in ample divisor'}\textup{(Theorem \ref{thm:pl-filp in ample divisor}, cf.\,\cite[Theorem 1.3]{hacon-witaszek20})}
Let $V$ be the spectrum of a complete discrete valuation ring.
Let $(X,S+A+B)$ be a dlt pair  such that $S$ is an anti-ample $\Q$-Cartier Weil divisor and $A$ is an ample  $\Q$-Cartier Weil divisor.
Let $f \colon X \to Z$ be a $(K_{X/V}+S+A+B)$-flipping contraction with $\rho(X/Z)=1$ to an affine $V$-scheme $Z$.
Assume that $(S^N,(1-\varepsilon)A_S+B_S)$ is globally $T$-regular for all $0<\varepsilon<1$ after localizing at all points of $f(\Exc(f))$.
Further assume that the ring 
\[
R(K_{S^N/V}+A_S+B_S):=\bigoplus_{m \in \Z_{\geq 0}} H^0(\sO_X(m(K_{S^N/V}+A_S+B_S)))
\]
is a finitely generated $\sO_Z$-algebra, where  $B_S:=\mathrm{Diff}_{S^N}(B)$ and $A_S:=A|_{S^N}$.
Then the flip of $f$ exists.
\end{thm}

\noindent
Theorem \ref{thm:pl-filp in ample divisor'} can be applied for three-dimensional pl-flips with ample divisor in the boundary.
%The existence of three-dimensional pl-flips with ample divisor in the boundary follows from Theorem \ref{thm:pl-filp in ample divisor'},
Indeed, the finite generation of $R(K_{S^N/V}+A_S+B_S)$ is a consequence of the MMP for excellent surfaces, and the global $T$-regularity of $(S^N,(1-\varepsilon)A_S+B_S)$ follows from the inversion of adjunction for global $T$-regularity (see Proposition \ref{prop:inversion of adjunction globally $T$-regular}) and an argument similar to the proof of \cite[Lemma 3.3]{hacon-witaszek19}.

\begin{rmk}
%After finishing this work, Jakub Witaszek taught us that he also write the paper about the MMP in mixed characteristic with Bhargav Bhatt, Linquan Ma, Zsolt Patakfalvi, Karl Schwede, Kevin Tucker, Joe Waldron (see \cite{bmpstww}).
After completing this work, Jakub Witaszek told us that he and Bhargav Bhatt, Linquan Ma, Zsolt Patakfalvi, Karl Schwede, Kevin Tucker, and Joe Waldron are also writing a paper on MMP in mixed characteristics (see \cite{bmpstww}).
In their article, they independently show that the MMP holds for threefolds whose each residue characteristic is greater than $5$.
They also define and study the notion of global $+$-regularity which is very closely related to our global $T$-regularity (see \cite[Lemma 4.7]{bmpstww}).
Furthermore, we can obtain an analog of Theorem \ref{thm:birational relative mmp} and Theorem \ref{thm:pl-filp in ample divisor'} in the case where $V$ is a spectrum of a regular excellent finite-dimensional domain by the argument in \cite{bmpstww}.
Indeed, we can show the results by using \cite[Proposition 3.6]{bmpstww} instead of Theorem \ref{thm:killing local cohomology} and the existence of regular alterations instead of the existence of Cohen-Macaulayfications.

\end{rmk}

%\begin{thm}\label{thm:relative mmp over cdvr}
%Let $V$ be the spectrum of excellent discrete valuation ring.
%Let $X$ be a flat $V$-variety of relative dimension two.
%Let $(X,\Delta)$ be a $\Q$-factorial dlt pair over $V$.
%Let $X \to Z$ be a projective morphism over $V$ to a quasi-projective $V$-variety $Z$.
%Assume that $\lfloor \Delta \rfloor$ contains the closed fiber $X_s$ as sets.
%Assume that \cite[Theorem 1.8]{witaszek20} holds.
%Then we can run a $(K_X+\Delta)$-MMP over $Z$ which terminates with a minimal model or a Mori fiber space.
%\end{thm}

\begin{ack}
The authors wish to express their gratitude to their supervisors Naoki Imai and Shunsuke Takagi for their encouragement, valuable advice, and suggestions.
%The first author is deeply grateful to his advisor Naoki Imai for his deep encouragement and helpful advice.
%The second author wishes to express his gratitude to his supervisor Shunsuke Takagi for his encouragement, valuable advice, and suggestions.
They thank Jakub Witaszek for valuable comments on the resolution in mixed characteristic and informing us of their work on the MMP in mixed characteristic. 
They are also grateful to  Hiromu Tanaka,  Kenta Sato, Kenta Hashizume, Yohsuke Matsuzawa, Ken Sato, Tatsuro Kawakami, Yoshinori Gongyo, Yuya Matsumoto, Tetsushi Ito, Teruhisa Koshikawa, and S\'{a}ndor Kov\'{a}cs for their helpful comments and suggestions. 
This work was supported by the Program for Leading Graduate Schools, MEXT, Japan.
The first author was supported by JSPS KAKENHI Grant number JP19J22795.
The second author was also supported by JSPS KAKENHI Grant number JP20J11886.
Moreover, the authors thank the referee for careful proofreading and many comments.
\end{ack}

\section{Preliminaries}
\subsection{Notations}
In this subsection, we summarize notations used in this paper.
\begin{itemize}
    \item We will freely use the notation and terminology in \cite{kollar-mori} and \cite{kollar13}.
    \item A morphism of schemes is \emph{alteration} if it is projective, surjective, and generically finite.
%    We note that every alteration is surjective.
    \item A Noetherian scheme $X$ is \emph{locally irreducible} if every connected component of $X$ is irreducible.
    \item A scheme $V$ is a \emph{Dedekind scheme} if $V$ is a Noetherian excellent $1$-dimensional regular scheme.
    \item For a Dedekind scheme $V$, we say $X$ is a \emph{variety over} $V$ or a $V$-\emph{variety} if $X$ is an integral scheme that is separated and of finite type over $V$. 
    We say $X$ is a \emph{curve} over $V$ or a $V$-\emph{curve} (respectively a \emph{surface} over $V$ or a $V$-\emph{surface}) if $X$ is a $V$-variety of (absolute) dimension one (respectively two).
    \item Let $V$ be a Dedekind scheme. 
    Let $\alpha \colon X \to V$ be a quasi-projective $V$-variety.
    The \emph{dualizing complex} $\omega^{\bullet}_{X/V}$ is defined by $\alpha^{!}\sO_V$, where $\alpha^!$ is defined as in \cite[Ch. III, Theorem 8.7]{hartshorne-duality}.
    The \emph{canonical sheaf} $\omega_{X/V}$ is defined by $(-d)$-th cohomology $h^{-d}(\omega^{\bullet}_{X/V})$ of the dualizing complex, where $d$ is the integer such that $(-d)$-th cohomology is the lowest non-zero cohomology of $\omega^{\bullet}_{X/V}$, thus there exists a natural map $\omega_{X/V}[d] \to \omega^{\bullet}_{X/V}$.
    We note that if $\alpha$ is flat, then $d$ coincides with the relative dimension of $X$ over $V$.
    If $X$ is normal, there is a Weil divisor $K_{X/V}$, called a \emph{canonical divisor}, such that $\omega_{X/V} \simeq \sO_X(K_{X/V})$.
    Note that $K_{X/V}$ is uniquely determined up to linear equivalence.
    We note that $\omega_{X/V}$ satisfies the condition $(S_2)$ by \cite[Lemma 1.3]{Hartshorne07}.
    If the image of $\alpha$ is a closed point, then we denote the induced morphisms by
    \[
    \xymatrix{
    X \ar[r]^-\beta & \Spec k \ar[r]^-\theta & V. 
    }
    \]
    Since $\theta^! \sO_V[1] \simeq k$, we have $\omega^{\bullet}_{X/V}[1] \simeq \omega^{\bullet}_{X/k}$.
    In particular, we have $\omega_{X/V} \simeq \omega_{X/k}$.
    In this case, we denote $\omega^{\bullet}_{X/V}[1], \omega_{X/V}, K_{X/V}$ by $\omega^{\bullet}_{X}, \omega_X, K_X$, respectively, for simplicity.
    \item Let $V$ be a Dedekind scheme.
    We say that $(X,\Delta)$ is a \emph{log pair} over $V$ if $X$ is a quasi-projective normal $V$-variety and $\Delta$ is an effective $\Q$-Weil divisor on $X$ such that $K_{X/V}+\Delta$ is $\Q$-Cartier.
    We will use the singularities of the MMP defined in \cite[Definition 2.8]{kollar13}.
    Let $S$ be a reduced divisor on $X$ such that $\Delta':=\Delta-S$ and $S$ have no common component.
    Then we denote the \emph{different} of the pair $(X,\Delta)$ by $\mathrm{Diff}_{S^N}(\Delta')$ defined by \cite[Definition 4.2]{kollar13}, where $S^N$ is the normalization of $S$.
    We will freely use the adjunction results in \cite[Section 4]{kollar13}.
    \item Let $V$ be a Dedekind scheme.
    Let $\pi \colon Y \to X$ be a proper morphism of $V$-schemes.
    The \emph{trace map} $R\pi_*\omega^{\bullet}_{Y/V} \to \omega^{\bullet}_{X/V}$ is defined as the following.
    Applying $R\mathscr{H}om(\ \underline{\hspace{4mm}}\ ,\omega^{\bullet}_{X/V})$ to the natural map $\sO_X \to R\pi_*\sO_Y$, we have
    \[
    R\mathscr{H}om(R\pi_*\sO_Y,\omega^{\bullet}_{X/V}) \to R\mathscr{H}om(\sO_X,\omega^{\bullet}_{Y/V}) \simeq \omega^{\bullet}_{X/V}.
    \]
    The left hand side is isomorphic to
    \[
    R\pi_*R\mathscr{H}om(\sO_Y,\omega^{\bullet}_{Y/V}) \simeq R\pi_*\omega^{\bullet}_{Y/V}
    \]
    by the Grothendieck duality, thus we obtain the trace map $R\pi_*\omega^{\bullet}_{Y/V} \to \omega^{\bullet}_{X/V}$.
    If $X$ is of relative dimension $d$ and $\pi$ is an alteration,
    taking the $(-d)$-th cohomology and the composition with the natural map, we obtain the map $\pi_*\omega_{Y/V} \to \omega_{X/V}$ is also called the trace map by abuse of notation.
    Let $D$ be a $\Q$-Cartier $\Q$-Weil divisor on $X$ and
    we further assume that $X$ and $Y$ are normal.
    Then we can extend the map $\pi_*\omega_{Y/V}(\lceil \pi^*D \rceil)|_U \to \omega_{X/V}(\lceil D \rceil)|_U$ on the regular locus $U$ of $X$ to the map $\pi_*\omega_{Y/V}(\lceil \pi^*D \rceil) \to \omega_{X/V}(\lceil D \rceil)$ on $X$ because $\omega_{X/V}(\lceil D \rceil)$ satisfies the condition $(S_2)$.
    \item Let $R$ be a discrete valuation ring and $V=\Spec R$.
    $R$ is of \emph{characteristic} $(0,p)$ if the fractional field $K$ is of characteristic zero and the residue field $k$ is of characteristic $p>0$.
    Let $X$ be a $V$-variety.
    The \emph{closed fiber} of $X$ is $X \times_V \Spec k$ denoted by $X_s$ and the \emph{generic fiber} of $X$ is $X \times_V \Spec K $ denoted by $X_\eta$.
\end{itemize}

\subsection{Negativity lemma and finite generation of the Picard rank}
In this subsection, we remark that the negativity lemma holds for the general setting.
Originally, the negativity lemma follows from the Bertini's theorem and the negativity lemma for surfaces.
However, in general setting, the Bertini type theorem is much harder.
Thus, we use the alternative proof by \cite{bdff}.

\begin{prop}\label{prop:negativity lemma}\textup{(Negativity lemma)}
Let $ \pi \colon Y \to X$ be a projective morphism of Noetherian normal schemes.
Let $D$ be an $\R$-Cartier $\R$-Weil $\pi$-nef divisor on $Y$.
If $\pi_*D \leq 0$, then $D \leq 0$. 
\end{prop}

\begin{proof}
The proof of \cite[Proposition 2.12]{bdff} also works in our setting.
Thus, we obtain an analogous statement of \cite[Proposition 2.12]{bdff}.
The negativity lemma follows from this statement.
\end{proof}

\begin{prop}\label{prop:finite generation of the Picard rank}
Let $X \rightarrow Y$ be a proper morphism of Noetherian schemes.
Then, the relative Picard number $\rank_{\R} N^{1}(X/Y)$ is finite, where $N^{1} (X/Y)$ is defined by
\[
(\Pic (X) / \textup{numerical equivalence over $Y$}) \otimes_{\Z}\R.
\]
\end{prop}

\begin{proof}
The proof of \cite[IV, \S 4]{Kleiman66} also works in our setting.
\end{proof}

\subsection{Base change}
In the proof of the existence of flips, we reduce to the case where $V$ is the spectrum of a complete discrete valuation ring with an infinite residue field.
To do this, we observe properties preserved under the base change via strictly henselization and completion.
We note that $\Q$-factoriality is not preserved under the above base change.

\begin{lem}\label{lem:topologically normal}
Let $V$ be an excellent Dedekind scheme.
Let $(X,\Delta)$ be a dlt pair.
Then every component $D$ of $\lfloor \Delta \rfloor$ is normal up to universal homeomorphism if $D$ is $\Q$-Cartier.
\end{lem}

\begin{proof}
We may assume that $D=\lfloor \Delta \rfloor$.
The assertion follows from the same argument as in the proof of \cite[Lemma 2.1]{hacon-witaszek20}.
\end{proof}

%\begin{lem}\label{lem:topologically normal lc centre}
%Let $V$ be an excellent Dedekind scheme.
%Let $(X,\Delta)$ be a three-dimensional dlt pair.
%Assume that every component of $\lfloor \Delta \rfloor$ is $\Q$-Cartier.
%Then every log canonical center of $(X,\Delta)$ is normal up to universal homeomorphism.
%\end{lem}

%\begin{proof}
%It follows from the same argument as in \cite[Remark 3.9]{hacon-witaszek19}.
%For codimension one log canonical center, it follows from Lemma \ref{lem:topologically normal}.
%For codimension three log canonical center, it is clear.
%Thus, it is enough to show that every codimension two log canonical center is normal up to universal homeomorphism.
%Take a codimension one log canonical center $C$, it is the irreducible component of $S_1 \cap S_2$, where $S_1$ and $S_2$ are components of $\lfloor \Delta \rfloor$.
%Let $\nu \colon S^N_1 \to S_1$ is the normalization and $C'$ is the preimage of $C$.
%$(X,\Delta)$ is simple normal crossing on the generic point of $C$, $C' \to C$ is finite and birational.
%By adjunction, $(S^N,D)$ is also dlt, where $D$ is the different with respect to $(X,\Delta)$.
%Since $C'$ is contained in $\lfloor D \rfloor$, $C'$ is normal.
%Thus, $\nu|_{C'} \colon C' \to C$ is normalization which is universal homeomorphism.
%\end{proof}

\begin{lem}\label{lem:base change}
Let $V$ be the spectrum of an excellent discrete valuation ring.
Let $\iota \colon V' \to V$ be the completion of the strict henselization of $V$.
Let $(X,\Delta)$ be a dlt pair over $V$ such that the dimension of $X$ is less than or equal to three.
Let $S$ be a component of $\lfloor \Delta \rfloor$ such that $S$ is $\Q$-Cartier.
Let $(X',\Delta')$ and $S'$ be the base change of $(X,\Delta)$ and $S$ via $\iota$, respectively.
Then $(X',\Delta')$ is dlt and $S'$ is locally irreducible.
In particular, every irreducible component of $S'$ is $\Q$-Cartier.
\end{lem}

\begin{proof}
We note that $\iota$ is the composition of formally \'{e}tale morphism and completion.
%By \cite[Lemma 2.5]{tanaka18}, relative canonical sheaf is compatible with taking base change.
Via both base changes, being dlt and normality are preserved, we note that a strict henselization of an excellent local ring is also excellent \cite[Crollary 5.6]{greco}.
Indeed, being dlt is characterized by using log resolution, which exists by \cite{cjs} and \cite{cossart-piltant}.
Furthermore, by \cite[Lemma 2.5]{tanaka18}, relative canonical sheaf is compatible with such base changes.
By Lemma \ref{lem:topologically normal}, the normalization $S^N \to S$ is a universal homeomorphism.
By the base change via $\iota$, we have $(S^N)' \to S'$, then $(S^N)'$ is also normal and this map is a universal homeomorphism.
Since normal schemes are locally irreducible and locally irreducible is preserved by homeomorphisms, $S'$ is also locally irreducible.
Since $S$ is $\Q$-Cartier, $S'$ is also $\Q$-Cartier.
Being $\Q$-Cartier is a local property and $S'$ is locally irreducible, so every irreducible component is also $\Q$-Cartier.
\end{proof}

\subsection{Alterations}

\begin{prop}\label{prop:extension of alterations}
Let $V$ be an excellent Dedekind scheme.
Let $X$ be a $V$-variety and $x \in X$ be a point.
Let $f_x \colon U' \to U:=\Spec \sO_{X,x}$ be an alteration from an integral scheme.
Then there exists an alteration $f \colon X' \to X$ of $V$-varieties such that 
we have the following Cartesian diagram:
\[
\xymatrix{
U' \ar[r] \ar[d] \ar@{}[rd]|\Box & U \ar[d] \\
X' \ar[r] & X.
}
\]
\end{prop}

\begin{proof}
Since $f_x$ is projective, $U'$ is embedded in some projective space $\P^N_U$.
It extends to some open subset of $X$, thus we may assume that $U$ is an open subset of $X$.
We denote the closure of $U'$ in $\P^N_X$ by $X'$.
Then the natural morphism $X' \to \P^N_X \to X$ is projective.
Since $U'$ is integral and $\dim U'= \dim X$, the morphism $X' \to X$ is also an alteration.
By the construction, this satisfies the desired conditions.
\end{proof}

\begin{prop}\label{prop:system of alterations}
Let $V$ be an excellent Dedekind scheme.
Let $X$ be a $V$-variety.
Let $f_i \colon Y_i \to X$ be alterations of $V$-varieties for $i=1, \ldots r$.
Then there exists an alteration $f \colon Y \to X$ of $V$-varieties which factors through $f_i$ for all $i$.
\end{prop}

\begin{proof}
By induction on $i$, we may assume that $i=2$.
We denote the fiber product of $Y_1$ and $Y_2$ over $X$ by $Y':=Y_1 \times_X Y_2$.
Take an irreducible component $Y$ of $Y'$ which dominates $X$, then $f \colon Y \to X$ is an alteration of $V$-varieties and $f$ factors through $f_i$.
\end{proof}

%\begin{prop}
%Let $V$ be an excellent Dedekind scheme.
%Let $X$ be a $V$-variety.
%Let $D$ be a Cartier divisor on $X$.
%Then for every positive integer $n$, there exists an alteration $f \colon Y \to X$ from a $V$-variety such that  $f^*D=nD'$ for some Cartier divisor $D'$ on $Y$.
%\end{prop}

%\begin{proof}
%Combining Proposition \ref{prop:extension of alterations} and Proposition \ref{prop:system of alterations}, we may assume that $D =\mathrm{div}(\phi)$ for some non-zero rational section $0 \neq \phi \in K(X)$.
%Taking the normalization of $X$ in $K(X)[\phi^{1/n}]$, denoted by $f' \colon Y' \to X$, we have $f^*D=n \mathrm{div}(\phi^{1/n})$. 
%\end{proof}

\begin{prop}\label{prop:extension of alterations of closed variety}
Let $V$ be an excellent Dedekind scheme.
Let $X$ be a $V$-variety and $S_1, \ldots, S_r$ be  closed sub-$V$-varieties of $X$.
Let $f_i \colon T_i \to S_i$ be an alteration of $V$-varieties.
Then there exists an alteration $g \colon Y \to X$ of  $V$-varieties and closed sub-$V$-varieties $S_{Y,1}, \ldots, S_{Y,r}$ of $Y$ with $g(S_{Y,i})=S_i$ for all $i$ such that the induced morphism $g_{S,i} \colon S_{Y,i} \to S_i$ factors through $f_i$.
\end{prop}

\begin{proof}
%By Proposition \ref{prop:system of alterations}, we may assume that $i=1$.
First, we consider the case of $r=1$.
We set $f:=f_1$, $S:=S_1$, and $T:=T_1$.
We denote the generic point of $S$ by $x$, then $\sO_{X,x}$ is a local domain with residue field $K(S)$.
Since $K(S) \subset K(T)$ is a finite extension of fields, it is a finite sequence of simple extensions.
Thus we have a finite extension of domains $\sO_{X,x} \to A$ such that the residue field of some maximal ideal of $A$ coincides with $K(T)$.
It extends to a finite surjective morphism $\pi \colon X' \to X$ such that it factors through $\sO_{X,x} \to A$ after localizing at $x$ by taking the normalization of $X$ in $K(A)$.
In particular, there exists a closed sub-$V$-variety $S'$ in $X'$ such that $\pi(S')=S$ and 
there exists a rational dominant map $S' \dashrightarrow T$ over $S$.
Then there exists a birational projective morphism $S'' \to S'$ with the blow-up ideal $\mathscr{I}_{S'}$ which eliminates the indeterminacy of $S' \dashrightarrow T$.
We take an ideal $\mathscr{I}$ of $\sO_X'$ with $\mathscr{I} \cdot \sO_S'=\mathscr{I}_{S'}$, and we denote the blow-up of $X'$ in the ideal $\mathscr{I}$ by $Y$.
Then $g \colon Y \to X$ satisfies the desired conditions.

Next, we consider the general case.
By the above case, we can find alterations $g_i \colon Y_i \to X$ and closed sub-$V$-schemes $S_{Y_i}$ of $Y_i$ such that $g(S_{Y_i})=S_i$ and $g_{Y_i} \colon S_{Y_i} \to S_i$ factors through $f_i$.
By Proposition \ref{prop:system of alterations}, there exists $g \colon Y \to X$ which factors through $g_i \colon Y_i \to X$.
Since $Y \to Y_i$ is surjective, there exists closed sub-$V$-variety $S_{Y,i}$ of $Y$ with $g(S_{Y,i})=S_{Y_i}$.
Therefore, $g$ and $S_{Y_i}$ satisfy the desired conditions.
\end{proof}

\begin{prop}\label{prop:index one cover1}
Let $V$ be an excellent Dedekind scheme.
Let $X$ be a $V$-variety and $D$ be a Cartier divisor on $X$.
Then for any integer $n$, there exists a finite cover $f \colon Y \to X$ such that $f^*D=nD'$ for some Cartier divisor $D'$ on $Y$.
\end{prop}

\begin{proof}
By Proposition \ref{prop:extension of alterations}, we may assume that $D =\mathrm{div}(\phi)$ for some non-zero section $0 \neq \phi \in K(X)$ and $X$ is affine by shrinking $X$.
Let $f \colon Y \to X$ be the normalization of $X$ in $K(X)[\phi^{1/n}]$, then $g^*D=n \mathrm{div}(\phi^{1/n})$.
\end{proof}

\begin{prop}\label{prop:index one cover}
Let $V$ be an excellent Dedekind scheme.
Let $X$ be a $V$-variety and $S \subset X$ be a closed $V$-variety.
Let $D$ be a Cartier divisor on $S$.
Then for every positive integer $n$, there exists an alteration $f \colon Y \to X$ from a $V$-variety and closed sub-$V$-variety $S_Y$ of $Y$ with $f(S_Y)=S$ such that  $f_S^*D=nD'$ for some Cartier divisor $D'$ on $S_Y$, where $f_S \colon S_Y \to S$ is the induced morphism.
\end{prop}

\begin{proof}
By Proposition \ref{prop:index one cover1}, there exists a finite cover $g \colon T \to S$ such that $g^*D=nD'$ for some Cartier divisor $D'$ on $T$.
The assertion follows from Proposition \ref{prop:extension of alterations of closed variety}.
%We take a lift $\phi \in sO_{X,x}$ whose natural image is $\phi_S$, where $x$ is the generic point of $S$.
%Let $f \colon Y \to X$ be the normalization of $X$ in $K(X)[\phi^{1/n}]$.
%We take a closed sub $V$-variety $S_Y$ of $Y$ such that $f(S_Y)=S$ and we denote the generic point of $S_Y$ by $y$, then $\phi^{1/n} \in \sO_{Y,y}$.
%Then the image $\phi_S^{1/n}$ of $\phi^{1/n}$ in $K(S_Y)$ is $n$-th root of the natural image of $\phi_S$ by the following commutative diagram.
%\[
%\xymatrix{
%\sO_{Y,y} \ar[d] \ar[r] & \sO_{X,x} \ar[d] \\
%K(S_Y) \ar[r] & K(S).
%}
%\]
%In particular, $f_S^*D=n \mathrm{div}(\phi^{1/n}_{S_Y})$.
\end{proof}

\begin{defn}\label{defn:semi-stable}\textup{(strictly semi-stable)}
\textup{ }
\begin{enumerate}
\item
Let $V$ be the spectrum of a discrete valuation ring $R$. Let $\varpi$ be a uniformizer of $R$.
A flat $V$-variety $X$ of relative dimension $n$ is called \emph{strictly semi-stable} if the following hold.
\begin{itemize}
    \item 
    The generic fiber $X_{\eta}$ is smooth, where $\eta \in V$ is the generic point.
    \item
    For any closed point $x$ in the special fiber $X_{s}$, there exists an open neighborhood $U$ of $x$ such that $U$ is \'{e}tale over the scheme
    \[
\Spec R[X_{0}, \ldots, X_{n}]/ (X_{0}\cdots X_{m} - \varpi)
    \]
for some $m \leq n$.
\end{itemize}
In this case, we also say that $(X,X_{s})$ is a strictly semi-stable pair.
As in \cite[2.16]{dejong} if $R$ has a perfect residue field, the above definition is equivalent to that $(X,X_{s})$ is a simple normal crossing pair.
\item 
Let $V$ be a Dedekind scheme.
An integral flat quasi-projective $V$-variety $X$ of relative dimension $n$ is called strictly semi-stable if $X_{\sO_{V,s}} \rightarrow \Spec \sO_{V,s}$ is strictly semi-stable for any closed point $s \in V$.
\end{enumerate}
%We note that a strictly semi-stable scheme over an excellent Dedekind scheme of relative dimension $2$ satisfies Assumption \ref{assump:semi-stable}.
\end{defn}

\begin{thm}\label{thm:semistable alteration}\textup{(\cite[Thorem 6.5]{dejong})}
Let $V$ be the spectrum of a complete discrete valuation ring.
Let $X$ be a flat $V$-variety, and let $Z \subset X$ be a proper closed subset.
Then there exists a finite surjective morphism $V' \to V$ and an alteration $\phi \colon X' \to X$ from a $V'$-variety such that the pair $(X',\phi^{-1}(Z)_{\mathrm{red}})$ is a strictly semi-stable pair, %by means of \cite[6.3]{dejong}, 
and in particular, $(X',\phi^{-1}(Z)_{\mathrm{red}})$ is a simple normal crossing pair.
\end{thm}

\begin{rmk}
If $X$ is not flat over $V$, then $X$ is defined over a field.
Thus, in this case, the last assertion of Theorem \ref{thm:semistable alteration} follows from \cite[Theorem 4.1]{dejong}.
\end{rmk}

\subsection{Adjunction and Bertini type theorem}
For the proof of the existence of flips, we discuss the adjunction of singularities related to local irreducibility.
The reader is referred to \cite[Section 4]{kollar13} for more details. 

\begin{prop}\label{prop:dlt adjunction to plt}
Let $V$ be an excellent Dedekind scheme.
Let $(X,S+A+B)$ be a dlt pair over $V$ such that $S$ and $A$ are locally irreducible Weil divisors and $\lfloor B \rfloor=0$.
Then $(S^N,\mathrm{Diff}_{S^N}(A+B))$ is plt.
\end{prop}

\begin{proof}
Let $\pi \colon S^N \to X$ be the composition of the normalization of $S$ and the closed immersion $S \to X$.
Let  $E$ be an exceptional prime divisor over $S^N$ centered at $x \in S^N$. 
If the log discrepancy $a_E(S^N,\mathrm{Diff}_{S^N}(A+B))$ is equal to $0$, then $\pi(x)$ is the generic point of a stratum of $\lfloor S+A \rfloor$.
Since $E$ is exceptional, $\pi(x)$ is contained in at least three components of $S+A$.
Since $S$ and $A$ are locally irreducible, this is the contradiction.
\end{proof}

%\begin{prop}\label{prop:different of plt or dlt pair}
%Let $V$ be an excellent Dedekind scheme.
%Let $(X,\Delta)$ be a log pair over $V$.
%\begin{itemize}
%    \item If $(X,\Delta)$ is dlt, then for every two irreducible components $S_1$, $S_2$ of $\lfloor \Delta \rfloor$ $S_1 \cap S_2$ is pure codimension two in $X$ or empty.
%    \item If $(X,\Delta)$ is plt, then $\lceil \Delta \rceil$ is locally irreducible.
%\end{itemize}
%\end{prop}

%\begin{proof}
%We take two irreducible components $S_1$, $S_2$ of $\lfloor \Delta \rfloor$ such that $S_1$ intersects $S_2$.
%We denote the normalization of $S_1$ by $S_1^N$.
%Shrinking $X$, we may assume that $S_1 \cap S_2$ is irreducible, it is denoted by $C$.
%If the codimension of $C$

%\end{proof}

\begin{lem}\label{lem:bertini for dlt surface}
Let $V$ be the spectrum of a discrete valuation ring with an infinite residue field.
Let $(X,\Delta)$ be a dlt surface over $V$.
Let $H$ be an ample $\Q$-Cartier divisor.
Then there exists an effective divisor $D \sim_{\Q} H$ such that $(X,\Delta+D)$ is also dlt.
\end{lem}

\begin{proof}
First, we note that if $X$ is not flat over $V$, then $X$ is a surface over the residue field $k$ of $V$.
Then the assertion is well-known.
Thus, we may assume that $X$ is flat over $V$.
We take a positive integer $m \geq 2$ such that $mH$ is very ample.
If there exists an effective divisor $D' \sim mH$ such that $(X,\Delta+D')$ is dlt, then $\frac{1}{m}D'$ is what we want.
Let $B_1$ be the sum of the components of $\Delta$ contained in the closed fiber $X_s$ and we write $B_2:=\Delta-B_1$.
Then $B_2$ and $X_s$ have no common components.
Let $\Sigma_s$ be the union of $B_2 \cap X_{s,\mathrm{red}}$ and the non-regular locus of $X_{s,\mathrm{red}}$.
We note that $\Sigma_s$ is a finite set as $X_s$ is one-dimensional.
Thus, for a general member $D$ of $|m(H|_{X_{s,\mathrm{red}}})|$, $D$ has no intersection with $\Sigma_s$, the pair $(X_{s,\mathrm{red}},D)$ is a simple normal crossing pair and $D$ has no common components with $B_1$ by \cite[Corollary 3.4.14]{fov-bertini}.
By the same argument, for a general member $D$ of $|mH|_{X_{\eta}}$, $D$ has no intersection with  the non-regular locus of $B_2|_{X_\eta}$.
By the proof of \cite[Theorem 0]{jannsen-saito}, we find an effective divisor $D'$ such that $D' \sim mH$ and $D'|_{X_\eta}$ and $D'|_{X_{s,\mathrm{red}}}$ satisfy the conditions as above.
Thus $(X,\Delta+D')$ is dlt.
Indeed, around the support of $D$, the pair $(X,\Delta+D)$ is a simple normal crossing pair.
\end{proof}

\subsection{Rational singularities}
In this subsection, we study properties of rational singularities.
We will use Proposition \ref{prop:rational singularities 1} and \ref{prop:rational singularities 2} in Section \ref{section:dedekind scheme} to prove the Cohen-Macaulayness of flips.
The reader is referred to \cite{kovacs17} for more details.
\begin{defn}\textup{(\cite{kovacs17})}
Let $X$ be a Noetherian excellent scheme. 
$X$ has \emph{rational singularities} if $X$ is normal and Cohen-Macaulay, and for every birational projective morphism $f \colon Y \to X$ from a normal  Cohen-Macaulay scheme $Y$, the natural morphism $\sO_X \to Rf_*\sO_Y$ is an isomorphism.
\end{defn}

\begin{prop}\label{prop:rational singularities 1}
Let $\pi \colon X \to Z$ be a projective birational morphism from a normal scheme $X$ to an excellent normal Cohen-Macaulay scheme $Z$ admitting a normalized dualizing complex $\omega^{\bullet}_Z$.
If $X$ has rational singularities and $\sO_Z \simeq R\pi_*\sO_X$, then $Z$ has rational singularities.
\end{prop}

\begin{proof}
We take a projective birational morphism $g \colon Y \to Z$ from a normal Cohen-Macaulay scheme $Y$.
We prove that the map $\sO_Z \to Rg_*\sO_Y$ is an isomorphism.
By \cite[Lemma 7.4]{kovacs17} and \cite[Theorem 8.6]{kovacs17}, we may assume that $g$ factors through $\pi$.
We denote the induced morphism by $f \colon Y \to X$.
We note that $f_*\sO_Y=\sO_X$ and $\pi_*\sO_X=\sO_Z$.
First we assume that $X$ has rational singularities and $\sO_Z \simeq R\pi_*\sO_X$, then we have $\sO_X \simeq Rf_*\sO_Y$.
By the spectral sequence, we have $\sO_Z \simeq R\pi_*\sO_X \simeq Rg_*\sO_Y$, so $Z$ has rational singularities.
\end{proof}

\begin{prop}\label{prop:rational singularities 2}
Let $\pi \colon X \to Z$ be a projective birational morphism from a normal  klt scheme $X$ to an excellent normal Cohen-Macaulay scheme $Z$ admitting a normalized dualizing complex $\omega^{\bullet}_Z$.
If $\dim \Exc (\pi) \leq 1$, $X$ has rational singularities except for finitely many closed points and $Z$ has rational singularities, then $X$ has rational singularities.
\end{prop}

\begin{proof}
We take a projective birational morphism $f \colon Y \to X$ from a normal Cohen-Macaulay scheme $Y$.
We put $g:\pi \circ f \colon Y \to Z$.
First, we prove that 
\[
\pi_*R^jf_*\sO_Y =0
\]
for all $j>0$.
We consider the spectral sequence
\[
E_2^{i,j}=R^i\pi_*R^jf_*\sO_Y \Longrightarrow E^{i+j}_{\infty}=R^{i+j}g_*\sO_Y.
\]
Since $\dim \Exc(\pi) \leq 1$, we have $E_2^{i,j}=0$ for $i>1$,
Therefore, we have $\pi_*R^if_*\sO_Y=E_2^{0,j}\simeq E_{\infty}^{0,j}$ for all $j$, and $E_{\infty}^{i,j}=0$ if $i \neq 0,1$.
Thus, we have 
\[
R^jg_*\sO_Y=E^j_{\infty}\simeq E_{\infty}^{0,j}=\pi_*R^jf_*\sO_Y
\]
for $j>1$.
Since $Z$ has rational singularities, it is zero.
Thus, it is enough to show that $E_{\infty}^{0,1}=0$.
Since $E^{i,j}_{\infty}=0$ if $i+j=1$ and $i \neq 0,1$, we obtain the exact sequence
\[
0 \to E_{\infty}^{0,1} \to E^{1}_{\infty} \to E_{\infty}^{1,0} \to 0.
\]
Since $Z$ has rational singularities, $E^1_{\infty}=R^1g_*\sO_Y=0$.
Therefore, we have $E_{\infty}^{0,1}=\pi_*R^1f_*\sO_Y=0$.

Since $X$ has rational singularities except for finite closed points, the support of $R^if_*\sO_Y$ is isolated, so we have $R^if_*\sO_Y=0$ for every $i>0$.
Thus, it is enough to show that $X$ is Cohen-Macaulay.
To do this, we prove the natural map $Rf_*\omega_Y \to \omega_X$ is an isomorphism.
Since $X$ has rational singularities except for finite closed points, $R^if_*\omega_Y$ has the isolated support for $i>0$ by \cite[Theorem 1.4]{kovacs17}.
Since $X$ is klt, we have $f_*\omega_Y \simeq \omega_X$.
Thus, by replacing the structure sheaves into canonical sheaves in the above argument, we have the natural isomorphism $Rf_*\omega_Y \simeq \omega_X$.
By the Grothendieck duality 
\[
Rf_* R\mathscr{H}om(\sO_Y,\omega^{\bullet}_{Y}) \simeq R\mathscr{H}om(Rf_*\sO_Y, \omega^{\bullet}_{X}),
\]
 we have
\[
Rf_*\omega^{\bullet}_Y \simeq \omega^{\bullet}_X.
\]
Since $Y$ is Cohen-Macaulay and $f$ is birational, $\omega^{\bullet}_Y$ is locally isomorphic to the shift of $\omega_Y$ over $X$, so $\omega^{\bullet}_X$ is locally isomorphic to the shift of $\omega_X$.
Thus $X$ is Cohen-Macaulay.
\end{proof}

\section{Existence of pl-flip with ample divisor in the boundary}\label{section:pl-flip}
In this section, we prove the existence of pl-flips with ample divisor in the boundary (cf.\,Theorem \ref{thm:pl-filp in ample divisor} and Corollary \ref{cor:existence of pl-flip threefold}).
In the first subsection, we study the vanishing theorem up to alterations.
Next, we introduce the notion of global $T$-regularity and study properties of it, for example, the adjunction and the inversion of adjunction.
Combining such arguments, we obtain the existence of flips in the special setting (cf.\,Theorem \ref{thm:pl-filp in ample divisor}).

In this section, we basically work over a scheme $V$, satisfying the following properties.

\begin{assump}\label{assump:cdvr}
$V$ is the spectrum of a complete discrete valuation ring of characteristic $(0,p)$.
\end{assump}

\begin{rmk}
Let $V$ be a scheme satisfying Assumption \ref{assump:cdvr}.
Let $X$ be a $V$-variety.
Then it is possible that $X$ is a variety over a field, and in such a case, $X=X_s$ and $X_\eta = \emptyset$, or $X=X_\eta$ and $X_s=\emptyset$. 
\end{rmk}

\subsection{Kodaira type vanishing up to alterations}

Bhatt \cite{bhatt20} proved the vanishing of local cohomology up to finite covers in mixed characteristic.
Using this theorem, we obtain the Kodaira type vanishing up to alterations for semiample and big divisors (Corollary \ref{cor:zero map}).
This theorem plays an essential role to prove the existence of flips.

\begin{thm}\label{thm:killing local cohomology}\textup{(\cite[Theorem 6.28]{bhatt20}, cf.\,\cite[Proposition 5.5.3]{bhatt10})}
Let $V$ be a scheme satisfying Assumption \ref{assump:cdvr}.
Let $f \colon X \to Z$ be a projective surjective morphism of $V$-varieties to an affine scheme $Z$.
Let $L$ be a semiample and $f$-big line bundle on $X$.
Let $z \in Z$ be a closed point with residue characteristic $p >0$.
Then there exists a finite surjective morphism $\pi \colon Y \to X$ from a $V$-variety such that
\[
 R\Gamma_zR\Gamma(X_{p},L^{-1}) \to R\Gamma_zR\Gamma(Y_{p},\pi^*L^{-1}),
\]
is zero on $h^i$ for all $i<\dim (X_p)$, where $X_p$ and $Y_p$ are closed subscheme of $X$ and $Y$ defined by $p=0$.
\end{thm}

\begin{proof}
If $X$ is flat over $V$, it is \cite[Theorem 6.28]{bhatt20}.
Otherwise, it follows from a similar argument to the argument in the proof of \cite[Proposition 5.5.3]{bhatt10}.
\end{proof}

\begin{rmk}
If $X$ is flat over $V$, the relative dimension of $X$ over $V$ coincides with $\dim (X_p)$ if the closed fiber is non-empty. 
\end{rmk}

\begin{prop}\label{prop:image is contained in varpi}
Let $V$ be a scheme satisfying Assumption \ref{assump:cdvr}.
Let $f \colon X \to Z$ be a projective surjective morphism from a flat $V$-variety $X$ to an affine flat $V$-variety $Z$.
We assume that the closed fiber of $Z \to V$ is non-empty.
Let $L$ be a semiample and $f$-big line bundle on $X$.
Then, for every positive integer $m$, there exists a finite surjective morphism $\pi \colon Y \to X$ from a $V$-variety such that the image of the following map
\[
\Tr^i_{\pi} \colon R^{i-d}\Gamma(\omega^{\bullet}_{Y/V}(\pi^*L)) \to R^{i-d}\Gamma(\omega^{\bullet}_{X/V}(L)),
\]
is contained in $\varpi^m R^{i-d}\Gamma(\omega^{\bullet}_{X/V}(L))$ for all $i>0$,
where $d$ is the relative dimension of $X$ over $Z$ and $\varpi$ is a uniformizer of $V$.
\end{prop}

\begin{proof}
We take a closed point $z \in Z_p$. %it is enough to show that the assertion holds at $z$ for some finite cover $\pi$ from a $V$-variety, because any finitely many finite covers from $V$-varieties are factored by some finite cover from a $V$-variety.
We set $A:=\sO_{Z_p,z}$.
Let $E$ be the injective hull of the residue field of $A$.
By Theorem \ref{thm:killing local cohomology}, there exists a finite surjective morphism $\pi \colon Y \to X$ from a $V$-variety such that
\[
h^iR\Gamma_zR\Gamma(X_p,L^{-1}) \to h^iR\Gamma_zR\Gamma(Y_p,\pi^*L^{-1})
\]
is zero for all $i<d$.
It is enough to show that $\pi$ satisfied the desired condition around $z$, because any finitely many finite covers from $V$-varieties are factored by some finite cover from a $V$-variety.
We take base changes via $\Spec{\sO_{Z,z}} \to Z$ and we use the same notations by abuse of notations.
We may replace $Z$ with $\Spec \sO_{Z,z}$.
By the local duality and the Grothendieck duality, we have
\[
R\Hom_A(R\Gamma_zR\Gamma(X_p,L^{-1}),E)) \simeq R\Gamma\omega^{\bullet}_{X_p/V}(L\widehat{)}[1],
\]
where the right hand side is the completion of the $1$-shift of $ R\Gamma\omega^{\bullet}_{X_p/V}(L)$.
By the same argument for $Y_p$, we have the following diagram;
\[
\xymatrix{
R\Hom_A(R\Gamma_zR\Gamma(Y_p,\pi^*L^{-1}),E) \ar[r] \ar[d]_{\simeq} & R\Hom_A(R\Gamma_zR\Gamma(X_p,L^{-1}),E) \ar[d]^{\simeq} \\
R\Gamma\omega^{\bullet}_{Y_p/V}(\pi^*L\widehat{)}[1] \ar[r] & R\Gamma\omega^{\bullet}_{X_p/V}(L\widehat{)}[1].
}
\]
Taking the $(i-d)$-th cohomology, we obtain that the trace map
\[
\Tr^i_{\pi_p} \colon R^{i+1-d}\Gamma(\omega^{\bullet}_{Y_p/V}(\pi^*L)) \to R^{i+1-d}\Gamma(\omega^{\bullet}_{X_p/V}(L)),
\]
is zero for all $i>0$.
Next, we consider the exact sequence
\[
0 \to \sO_X \to \sO_X \to \sO_{X_p} \to 0,
\]
and applying $R\Gamma R\mathscr{H}om(-, \omega^{\bullet}_{X/V}(L))$ and taking $(i-d)$-th cohomology, we have the exact sequence
\[
R^{i-d}\Gamma(\omega^{\bullet}_{X/V}(L)) \to R^{i-d}\Gamma(\omega^{\bullet}_{X/V}(L)) \to  R^{i+1-d}\Gamma(\omega^{\bullet}_{X_p/V}(L)).
\]
Thus, we have the commutative diagram of exact sequences
\[
\xymatrix{
R^{i-d}\Gamma(\omega^{\bullet}_{X/V}(L)) \ar[r]^{\cdot p} & R^{i-d}\Gamma(\omega^{\bullet}_{X/V}(L)) \ar[r] & R^{i+1-d}\Gamma(\omega^{\bullet}_{X_p/V}(L)) \\
R^{i-d}\Gamma(\omega^{\bullet}_{Y/V}(\pi^*L)) \ar[r]^{\cdot p} \ar[u] & R^{i-d}\Gamma(\omega^{\bullet}_{Y/V}(\pi^*L)) \ar[r] \ar[u] & R^{i+1-d}\Gamma(\omega^{\bullet}_{Y_p/V}(\pi^*L)) \ar[u]_{0\mathrm{-map}}.
}
\]
Therefore,  $\mathrm{Im}(\Tr^i_\pi)$ is contained in $pR^{i-d}\Gamma(\omega^{\bullet}_{X/V}(L))$.
By the same argument for $Y$, there exists a finite surjective morphism $g \colon W \to Y$ such that $\mathrm{Im}(\Tr^i_g)$ is contained in $pR^{i-d}\Gamma(\omega^{\bullet}_{Y/V}(\pi^*L))$, thus we have
\[
\mathrm{Im}(\Tr^i_{\pi \circ g}) \subset p^2 R^{i-d}\Gamma(\omega^{\bullet}_{X/V}(L)). 
\]
Repeating such a process, we obtain a finite surjective morphism of $V$-varieties such that the image of the trace map is contained in $p^mR^{i-d}\Gamma(\omega^{\bullet}_{X/V}(L))$.
Since $p$ is a multiple of $\varpi$, we have the assertion. 
\end{proof}

\begin{cor}\label{cor:zero map}\textup{(cf.\,\cite[Theorem 5.5]{bst})}
Let $V$ be a scheme satisfying Assumption \ref{assump:cdvr}.
Let $f \colon X \to Z$ be a projective surjective morphism from a normal flat $V$-variety $X$ to an affine flat $V$-variety $Z$.
Let $D$ be a semiample $\pi$-big Cartier divisor on $X$.
Then there exists an alteration $\pi \colon Y \to X$ such that the trace map
\[
\Tr^i_\pi \colon R^{i-d}\Gamma(\omega^{\bullet}_{Y/V}( \pi^*D )) \to R^{i-d}\Gamma(\omega^{\bullet}_{X/V}( D ))
\]
is zero for all $i> 0$, where $d$ is the relative dimension of $X$ over $V$.
Furthermore, if  $X$ is Cohen-Macaulay, there exists an alteration $\pi \colon Y \to X$ such that the trace map
\[
\Tr^i_\pi \colon H^i(\omega_{Y/V}( \pi^*D )) \to H^i(\omega_{X/V}( D ))
\]
is zero for all $i>0$.
\end{cor}

\begin{proof}
We may assume that  $X$ is regular by taking an alteration by Theorem \ref{thm:semistable alteration}.
Then the generic fiber $X_\eta$ is a variety over a field of characteristic zero and $D_\eta$ is semiample and big over $Z$, thus we have 
\[
R^{i-d}\Gamma(\omega^{\bullet}_{X/V}(D))_\eta=0
\]
by Kawamata-Viehweg vanishing.
Therefore, we obtain the case where the closed fiber of $X$ is empty.
Thus, we assume that the closed fiber of $X$ is non-empty.
Furthermore, $R^{i-d}\Gamma(\omega^{\bullet}_{X/V}( D ))$ is a $\varpi$-torsion finite $\sO_Z$-module.
Thus, for some positive integer $m$, we have 
\[
\varpi^m R^{i-d}\Gamma(\omega^{\bullet}_{X/V}( D))=0
\]
By Proposition \ref{prop:image is contained in varpi}, 
there exists a finite surjective morphism $\pi \colon Y \to X$ from a $V$-variety such that the image of the following map
\[
\Tr^i_{\pi} \colon R^{i-d}\Gamma(\omega^{\bullet}_{Y/V}(\pi^*D)) \to R^{i-d}\Gamma(\omega^{\bullet}_{X/V}(D)),
\]
is contained in $\varpi^m R^{i-d}\Gamma(\omega^{\bullet}_{X/V}(D))=0$ for all $i>0$.
Thus, we obtain the first assertion.

Next, we assume that $X$ is Cohen-Macaulay.
By the first assertion, there exists an alteration $\pi \colon Y \to X$ such that the trace map
\[
\Tr^i_\pi \colon R^{i-d}\Gamma(\omega^{\bullet}_{Y/V}( \pi^*D )) \to R^{i-d}\Gamma(\omega^{\bullet}_{X/V}( D ))
\]
is zero for all $i> 0$.
Since $X$ is Cohen-Macaulay, we have $R^{i-d}\Gamma(\omega^{\bullet}_{X/V}(D))\simeq H^i(\omega_{X/V}(D))$.
Thus, the trace map
\[
H^i(\omega_{Y/V}(\pi^*D)) \to R\Gamma^{i-d}(\omega^{\bullet}_{Y/V}(\pi^*D)) \to H^i(\omega_{X/V}(D))
\]
is zero for all $i>0$.
\end{proof}

\begin{rmk}\label{rmk:zero map}
In positive characteristic, an analogous statement of Corollary \ref{cor:zero map} holds and
we can take $Y$ in Corollary \ref{cor:zero map} as a finite cover by \cite[Theorem 5.5]{bst}.
\end{rmk}

\subsection{Global $T$-regularity}
In positive characteristic, global $F$-regularity is important in the proof of the existence of flips (see \cite{hacon-xu}, \cite{hacon-witaszek19}, \cite{hacon-witaszek20}).
In mixed characteristic, we use  global $T$-regularity instead of it.

\begin{defn}
Let $\pi \colon Y \to X$ be an alteration of normal schemes.
\begin{itemize}
    \item For a prime divisor $E$ on $X$, a prime divisor $E_Y$ is called a \emph{strict transform of $E$} if $\pi(E_Y)=E$.
    \item For an $\R$-Weil divisor $D=\sum_i a_i E_i$ , where $E_i$ is a prime divisor, an $\R$-Weil divisor $D_Y$ is called a \emph{strict transform of} $D$ if $D_Y$ is denoted by $D_Y=\sum_i a_i E_{Y,i}$ such that each $E_{Y,i}$ is a strict transform of $E_i$.
\end{itemize}
\end{defn}

\begin{rmk}
\textup{}
\begin{itemize}
    \item Since $\pi$ is an alteration, $\pi|_{E_Y}$ is also an alteration.
    \item If $D$ is $\Q$-Weil or $\Z$-Weil, then so is $D_{Y}$
    \item If $D$ is a locally irreducible reduced divisor, then so is $D_Y$.
    \item If $\pi$ is birational, then $D_Y$ is the strict transform of $D$ in the usual sense
\end{itemize}
\end{rmk}

\begin{defn}\textup{(Globally $T$-regular, Purely globally $T$-regular)}
Let $V$ be an excellent Dedekind scheme.
Let $(X,\Delta)$ be a log pair, or a localization of a log pair over $V$.
\begin{itemize}
    \item Let $L$ be a $\Q$-Cartier $\Q$-Weil divisor on $X$. Then the submodule $T^0(X,\Delta;L)$ of $H^0(\sO_X(\lceil L \rceil))$ is defined by
    \[
    \bigcap_{\pi \colon Y \to X} \mathrm{Im}(H^0(\omega_{Y/V}(\lceil \pi^*(L-K_{X/V}-\Delta) \rceil))  \to H^0(\sO_X(\lceil L \rceil))),
    \]
    where the map is a composition of  a trace map and the natural injection and $\pi$ runs over all alterations from a normal $V$-variety.
    \item The pair $(X,\Delta)$ is \emph{globally $T$-regular} (globally Trace-regular) if for every alteration $\pi \colon Y \to X$ from a normal $V$-variety, the trace map
    \begin{equation}\label{equation:globally T-regulality}
     H^0(\omega_{Y/V}(\lceil -\pi^*(K_{X/V}+\Delta) \rceil)) \to H^0(\sO_X) 
    \end{equation}
    is surjective, that is, $T^0(X,\Delta):=T^0(X,\Delta;0)=H^0(\sO_X)$.
    \item We set $\lfloor \Delta \rfloor=S$.
    Then $(X,\Delta)$ is \emph{purely globally $T$-regular} if for every alteration $ \pi \colon Y \to X$ from a normal $V$-variety and every strict transform $S_Y$ of $S$ via $\pi$, the trace map
    \begin{equation}\label{equation:purely globally T-regulality}
        H^0(\omega_{Y/V}(S_Y+\lceil -\pi^*(K_{X/V}+\Delta) \rceil))  \to H^0(\sO_X) 
    \end{equation}
    is surjective. 
     We note that the map in (\ref{equation:purely globally T-regulality}) is well-defined.
    Indeed, it is enough to show that the map
    \[
    \pi_*\omega_{Y/V}(S_Y+\lceil -\pi^*(K_{X/V}+\Delta) \rceil)  \to \sO_X
    \]
    is well-defined on the regular locus of $X$.
    On the regular locus of $X$, $S_Y \leq \pi^*S$, so the left hand side is contained in
    \[
    \omega_Y(\lceil -\pi^*(K_{X/V}+\Delta') \rceil),
    \]
    where $\Delta':=\Delta-S$.
    \item If $X$ is affine, we say that $(X,\Delta)$ is \emph{$T$-regular} (resp. \emph{purely $T$-regular}) if it is globally $T$-regular (resp. purely globally $T$-regular).
    \item Let $f \colon X \to Z$ be a quasi-projective morphism of $V$-varieties.
    $(X,\Delta)$ is globally $T$-regular (resp. purely globally $T$-regular) over $z \in Z$ if $(X,\Delta)$ is globally $T$-regular (resp. purely globally $T$-regular) after localizing at $z$.
\end{itemize}
\end{defn}

\begin{prop}\label{prop:corresponding to finite cover in positive chara}\textup{(\cite[Theorem 6.7]{bst})}
Let $X$ be a quasi-projective variety over an $F$-finite field.
Let $(X,\Delta)$ be a log pair and $L$ a $\Q$-Cartier $\Q$-Weil divisor on $X$.
Then
\[
T^0(X,\Delta;L)=\bigcap_{\pi} \mathrm{Im}(H^0(\omega_Y(\lceil \pi^*(L-K_X-\Delta) \rceil))  \to H^0(\sO_X(\lceil L \rceil))),
\]
where $\pi$ runs over all finite covers from a normal $V$-variety.
\end{prop}

\begin{rmk}
In \cite{bst}, $T^0$ is defined as in Proposition \ref{prop:corresponding to finite cover in positive chara}, and they proved the equivalence of the definitions in positive characteristic.
However, in characteristic zero, it does not hold in general.
Indeed, let $X$ be a normal non-klt variety over a field of characteristic zero.
By Proposition \ref{prop:klt and plt by globally $T$-regular}, $X$ is not globally $T$-regular, so $T^0$ is not $H^0(\sO_X)$.
On the other hand, since every finite cover of normal variety splits, the right hand side in Proposition \ref{prop:corresponding to finite cover in positive chara} is $H^0(\sO_X)$.
\end{rmk}

\begin{prop}\label{prop:localization globally $T$-regular}
Let $V$ be a scheme satisfying Assumption \ref{assump:cdvr}.
Let $f \colon X \to Z$ be a projective surjective morphism of $V$-varieties.
Let $(X,\Delta)$ be log pair over $V$.
If $(X,\Delta)$ is globally $T$-regular, then so is $(X',\Delta')$, where $(X',\Delta')$ is one of the following.
\begin{itemize}
    \item A restriction of $(X,\Delta)$ over  on some  open subset $X'$ of $X$, or
    \item A localization of $(X,\Delta)$ at some point of $Z$.
\end{itemize}
Furthermore, the same assertion holds for purely globally $T$-regular log pairs.
\end{prop}

\begin{proof}
We can extend every alteration of $X'$ to one of $X$ by Proposition \ref{prop:extension of alterations}, and the surjectivity is equivalent to the property that image of the trace map contains $1 \in H^0(\sO_X)$.
It does not change after localization or restriction.
\end{proof}

\begin{prop}\label{prop:klt and plt by globally $T$-regular}
Let $V$ be an excellent Dedekind scheme.
Let $(X,\Delta)$ be a log pair over $V$.
If $(X,\Delta)$ is globally $T$-regular, then it is klt, and in particular, $\lfloor \Delta \rfloor=0$.
If  $(X,\Delta)$ is purely globally $T$-regular, then it is plt, and in particular, $\lfloor \Delta \rfloor$ is reduced.
\end{prop}

\begin{proof}
By Proposition \ref{prop:localization globally $T$-regular}, we may assume that $X$ is affine.
We take a birational projective morphism $\pi \colon Y \to X$ from a normal $V$-variety.
If $(X,\Delta)$ is globally $T$-regular, then  we have
$\pi_*\omega_{Y/V}(\lceil -\pi^*(K_{X/V}+\Delta) \rceil)=\sO_X$, as trace map is injective in the birational case.
Thus, $(X,\Delta)$ is klt.
If $(X,\Delta)$ is purely globally $T$-regular, we denote $S:=\lfloor \Delta \rfloor$, then we have $\omega_{Y/V}(S_Y+\lceil -\pi^*(K_{X/V}+\Delta) \rceil)=\sO_X$, where $S_Y$ is the strict transform of $S$.
Thus, $(X,\Delta)$ is plt.
\end{proof}

\begin{prop}\label{prop:purely globally $T$-regular to globally $T$-regular}
Let $V$ be an excellent Dedekind scheme.
Let $(X,S+\Delta)$ be a log pair with $\lfloor S+\Delta \rfloor=S$.
Assume that $(X,S+\Delta)$ is purely globally $T$-regular and $S$ is $\Q$-Cartier.
Then $(X,(1-\varepsilon)S+\Delta)$ is globally $T$-regular for all $0< \varepsilon <1$ with $\varepsilon \in \Q$.
\end{prop}

\begin{proof}
We fix a $0<\varepsilon<1$ and a positive integer $m$ with $m\varepsilon>1$.
We note that for every alteration $\mu \colon W \to X$, we can find an alteration $\pi \colon Y \to X$ and a strict transform $S_Y$ such that $\pi^*S \geq mS_Y$ and $\pi$ factors through $\mu \colon W \to X$ by applying Proposition \ref{prop:index one cover1} to $W$.
Indeed, we take a positive integer $r$ such that $r\mu^*S$ is Cartier.
By Proposition \ref{prop:index one cover1}, we take a finite cover $\pi' \colon Y \to W$ such that $r\pi^*S=mrD$ for some Cartier divisor $D$ on $Y$, where $\pi=\pi'\circ \mu$.
We take a strict transform $S_Y$ of $S$ with $S_Y \leq D$.
Then we have $mS_Y \leq mD=\pi^*S$.
We take an alteration $\pi \colon Y \to X$ from a normal $V$-variety and a strict transform $S_Y$ of $S$ such that $\pi^*S \geq mS_Y$ and $\pi^*(K_X+\Delta)$ is Cartier, the existence follows from Proposition \ref{prop:index one cover1}.
It is enough to show that the trace map
\[
H^0(\omega_{Y/V}(\lceil -\pi^*(K_{X/V}+(1-\varepsilon)S+D) \rceil )) \to H^0(\sO_X)
\]
is surjective.
Since $\pi^*(1-\varepsilon)S \leq \pi^*S-\varepsilon mS_Y \leq \pi^*S-S_Y$,
the left hand side is larger than the left hand side of the map (\ref{equation:purely globally T-regulality}) .
Thus, if $(X,S+\Delta)$ is purely globally $T$-regular, then $(X,(1-\varepsilon)S_Y+\Delta)$ is globally $T$-regular.
\end{proof}

\begin{prop}\label{prop:first properties of glovally splinter}
Let $V$ be a scheme satisfying Assumption \ref{assump:cdvr}.
Let $f \colon X \to Z$ be a projective surjective morphism from a normal $V$-variety $X$ to an affine $V$-variety $Z$.
Let $(X,\Delta)$ be a globally $T$-regular log pair over $V$.
Then the following conditions hold.
\begin{enumerate}
    \item For every alteration $\pi \colon Y \to X$ from a normal $V$-variety, the trace map
    \[
    \pi_*\omega_{Y/V}(\lceil -\pi^*(K_{X/V}+\Delta) \rceil) \to \sO_X
    \]
    is a splitting surjection.
    \item For every finite surjective morphism $\pi \colon Y \to X$ from a normal $V$-variety, the canonical map
    \[
    \sO_X \to \pi_*\sO_Y(\lfloor \pi^*\Delta \rfloor)
    \]
    splits.
    %\item If $X$ is a variety over an $F$-finite field of positive characteristic and $X$ is affine, then $(X,\Delta)$ is strongly $F$-regular.
    %\item $\sO_X(D)$ is maximal Cohen-Macaulay for all $\Q$-Cartier Weil divisor $D$ on $X$.
    %\item For semiample and $f$-big $\Q$-Cartier Weil divisor $D$ on $X$, the equation
    %\[
   % H^i(\omega_X(D))=0
   % \]
   % holds for all $i>0$.
\end{enumerate}
\end{prop}

\begin{proof}
By the equation (\ref{equation:globally T-regulality}), there exists a global section 
\[
\alpha \in H^0(\omega_{Y/V}(\lceil -\pi^*(K_{X/V}+\Delta) \rceil))
\]
mapped to $1 \in H^0(\sO_X)$.
The section $\alpha$ defines a morphism 
\[
\sO_X \to \pi_*\omega_{Y/V}(\lceil -\pi^*(K_{X/V}+\Delta) \rceil)
\]
mapping $1$ to $\alpha$, thus it gives a splitting and we obtain $(1)$.
Applying $\mathscr{H}om(\ \underline{\hspace{4mm}}\ ,\sO_X)$ for the trace map in $(1)$, we obtain the canonical map in $(2)$ when $\pi$ is finite surjective, thus we have $(2)$.
\end{proof}

\begin{lem}\label{lem:splitting}
Let $V$ be a scheme satisfying Assumption \ref{assump:cdvr}.
Let $(X,\Delta)$ be a globally $T$-regular log pair over $V$.
Let $D$ be a Weil $\Q$-Cartier divisor on $X$.
Let $f \colon Y \to X$ be a finite surjective morphism from a normal $V$-variety such that $f^*D$ is Cartier.
Then $\sO_X(D) \to f_*\sO_Y(f^*D)$ splits.
\end{lem}

\begin{proof}
By Proposition \ref{prop:first properties of glovally splinter} $(2)$, $\sO_X \to f_*\sO_Y$ splits.
Thus, $\sO_X(D) \to f_*\sO_Y(f^*D)$ splits on the regular locus of $X$.
Since $\sO_X(D)$ and $f_*\sO_X(f^*(D))$ are reflexive, the splitting extends to the splitting on $X$.
\end{proof}

\begin{prop}\label{prop:vanishing for T-regular}
Let $V$ be a scheme satisfying Assumption \ref{assump:cdvr}.
Let $f \colon X \to Z$ be a projective surjective morphism from a normal $V$-variety $X$ to an affine $V$-variety $Z$.
Let $(X,\Delta)$ be a globally $T$-regular log pair over $V$.
Let $D$ be a Weil $\Q$-Cartier divisor.
Then the following hold.
\begin{enumerate}
    \item If $D$ is semiample, then $(R^if_*(\sO_X(D)))_z=0$ for all $z \in Z_s$ and $i>0$, where $Z_s$ is a closed fiber.
    \item $\sO_X(D)$ is maximal Cohen-Macaulay.
    \item If $D$ is semiample and $f$-big, then  $H^i(\omega_{X/V}(D))=0$  for all $i>0$.
\end{enumerate}
\end{prop}

\begin{proof}
By \cite[Theorem 6.28]{bhatt20}, there exists a finite surjective morphism $ \pi \colon Y \to X$ such that $\pi^*D$ is Cartier and
\[
H^i(\sO_{X_p}(D)) \to H^i(\sO_{Y_p}(\pi^*D))
\]
is zero for all $i>0$.
Indeed, the assertion is reduced to the case where $D$ is Cartier.
By Lemma \ref{lem:splitting}, we have $H^i(\sO_{X_p}(D))=0$.
Thus, we have the map
\[
\xymatrix{
H^i(\sO_X(D)) \ar[r]^{\cdot p} & H^i(\sO_X(D))
}
\]
is surjective for all $i>0$.
Thus, By Nakayama's Lemma, $H^i(\sO_X(D))=0$ at all points of $Z_s$.

Next, we consider $(2)$.
Since the assertion $(2)$ is a local question, we may assume that $f$ is the identity map.
By \cite[Theorem 6.28]{bhatt20}, for a closed point $x \in X_s$, there exists a finite surjective morphism $\pi \colon Y \to X$ such that $\pi^*D$ is Cartier and the map
\[
R^i\Gamma_xR\Gamma(\sO_{X_p}(-D)) \to R^i\Gamma_xR\Gamma(\sO_{Y_p}(-\pi^*D))
\]
is zero for all $i < \dim X_p$.
By Lemma \ref{lem:splitting}, we have $R^i\Gamma_xR\Gamma(\sO_{X_p}(-D))=0$.
Therefore, $\sO_X(D)$ is maximal Cohen-Macaulay around $X_s$.
Since $(X,\Delta)$ is klt, $\sO_X(D)$ is maximal Cohen-Macaulay on the generic fiber by \cite[Corollary 5.25]{kollar-mori}, thus we obtain the assertion $(2)$.

Finally, we consider $(3)$.
By \cite[Theorem 6.28]{bhatt20}, for a closed point $z \in Z$, there exists a finite surjective morphism $\pi \colon Y \to X$ such that $\pi^*D$ is Cartier and the map
\[
R^i\Gamma_zR\Gamma(\sO_{X_p}(-D)) \to R^i\Gamma_zR\Gamma(\sO_{Y_p}(-\pi^*D))
\]
is zero for all $i < \dim X_p$.
By Lemma \ref{lem:splitting}, we have $R^i\Gamma_zR\Gamma(\sO_{X_p}(-D))=0$.
By the argument of the proof of Proposition \ref{prop:image is contained in varpi}, we have
\[
H^i(\mathscr{H}om(\sO_{X_p}(-D),\omega_{X_p/V}))=0
\]
for all $i>0$.
Thus, we have that
\[
\xymatrix{
H^i(\omega_{X/V}(D)) \ar[r]^{\cdot p} & H^i(\omega_{X/V}(D))
}
\]
is surjective for all $i>0$.
Thus, $H^i(\omega_{X/V}(D))=0$ around $Z_s$.
By Kawamata-Viehweg vanishing for the generic fiber $X_\eta$, the vanishing holds on $Z_\eta$.
Thus we obtain the assertion $(3)$.
\end{proof}

\begin{prop}\label{prop:crepant globally $T$-regular}
Let $V$ be an excellent Dedekind scheme.
Let $g \colon Y \to X$ be a projective birational morphism of normal $V$-varieties.
Let $(X,\Delta)$ and $(Y,\Gamma)$ be log pairs such that $g^*(K_{X/V}+\Delta)=K_{Y/V}+\Gamma$.
Then $(X,\Delta)$ is globally $T$-regular if and only if $(Y,\Gamma)$ is globally $T$-regular.
The same assertion holds for purely globally $T$-regular case if $\lfloor \Gamma \rfloor$ is the strict transform of $\lfloor \Delta \rfloor$.
\end{prop}

\begin{proof}
First, we consider the globally $T$-regular case.
We note that $(X,\Delta)$ is klt if and only if $(Y,\Gamma)$ is klt.
By Proposition \ref{prop:klt and plt by globally $T$-regular}, we may assume that $(X,\Delta)$ and $(Y,\Gamma)$ are klt, and in particular, $\lfloor \Gamma \rfloor =0$.
Thus, the trace map coincides with the following isomorphism
\[
H^0(\omega_{Y/V}(\lceil -g^*(K_{X/V}+\Delta) \rceil)=H^0(\sO_Y) \simeq H^0(\sO_X). 
\]
Take an alteration $\pi \colon W \to Y$.
Then the composition of trace maps
\[
H^0(\omega_{W/V}(\lceil -\pi^*(K_{Y/V}+\Gamma) \rceil)) \to H^0(\sO_Y) \simeq H^0(\sO_X)
\]
is the trace map with respect to $g \circ \pi$.
Then the surjectivities of two trace maps are equivalent to each other.

Next, we consider the purely globally $T$-regular case.
We denote $\lfloor \Delta \rfloor$ and $\lfloor \Gamma \rfloor$ by $S$ and $T$, respectively.
Then the corresponding trace map coincides with natural isomorphic $H^0(\sO_Y) \to H^0(\sO_X)$.
Indeed, $\lceil K_{Y/V}+T-g^*(K_{X/V}+\Delta) \rceil=\lceil T-\Gamma \rceil =0$.
Thus, by the same argument as above, we obtain the equivalence.
\end{proof}

\begin{lem}\label{lem:good alteration}
Let $V$ be a scheme satisfying Assumption \ref{assump:cdvr}.
Let $(X,S+B)$ be a log pair over $V$ such that $S$ is a reduced divisor  and $S$ and $B$ have no common components.
Let $\pi \colon X' \to X$ be an alteration of normal $V$-varieties and $S'$ be a strict transform of $S$ on $X'$.
Then there exists an alteration $f \colon Y \to X$ from a normal $V$-variety and a strict transform $S_Y$ of $S$ such that the following hold.
\begin{itemize}
    \item $f$ factors through $\pi$ and $S_Y$ is a strict transform of $S'$,
    \item $S_Y$ is locally irreducible,
    \item $(Y,S_Y)$ is a simple normal crossing pair, 
    \item $f^*(K_{X/V}+S+B)$ and $f_S^*(K_{S^N/V}+B_S)$ are Cartier, and
    \item The following diagram commutes.
    \[
    \xymatrix{
    f_*\omega_{Y/V}(S_Y-f^*(K_{X/V}+S+B)) \ar[r] \ar[d] & \sO_X \ar[d] \\
    j_*f_{S,*}\omega_{S_Y/V}(-f_S^*(K_{S^N/V}+B_S)) \ar[r] & j_*\sO_{S^N},
    }
    \]
\end{itemize}
where $B_S$ is the different of $(X,S+B)$, $f_S \colon S_Y \to S^N$ is the induced morphism, $j \colon S^N \to X$ is the composition of the normalization and the inclusion, the left vertical map follows from the adjunction formula and horizontal maps are induced by trace maps.
\end{lem}

\begin{proof}
By the definition of different, we have $(K_{X/V}+S+B)|_{S^N} \sim_{\Q} K_{S^N/V}+B_S$.
By Theorem \ref{thm:semistable alteration}, there exists an alteration $f \colon Y \to X$ and a strict transform $S_Y$ of $S$ such that  $(Y,S_Y)$ is a simple normal crossing pair and $f$ factors through $\pi$.
By taking blowing up along the stratum of $S_Y$, we may assume that $S_Y$ is locally irreducible, thus $S_Y$ is regular.
In particular, $S_Y \to S$ factors through the normalization $S^N \to S$ denoted by $f_S \colon S_Y \to S^N$.
By Proposition \ref{prop:index one cover}, we may assume that $f^*(K_{X/V}+S+B)$ and $f^*(K_{S^N/V}+B_S)$ is Cartier and 
\[
f^*(K_{X/V}+S+B)|_{S_Y} \sim f_S^*((K_{X/V}+S+B)|_{S^N}) \sim f_S^*(K_{S^N/V}+B_S).
\]
Thus, we define the morphism $\omega_{Y/V}(S_Y-f^*(K_{X/V}+S+B)) \to \omega_{S_Y/V}(-f_S^*(K_{S^N/V}+B_S))$ induced by the adjunction formula $\omega_{Y/V}(S_Y)|_{S_Y} \simeq \omega_{S_Y/V}$.
Then it is enough to show that the diagram
\[
\xymatrix{
f_*\omega_{Y/V}(S_Y-f^*(K_{X/V}+S+B)) \ar[r] \ar[d] & \sO_X \ar[d] \\
j_*f_{S,*}\omega_{S_Y/V}(-f_S^*(K_{S^N/V}+B_S)) \ar[r] & j_*\sO_{S^N},
}
\]
commutes.
Since $\mathscr{H}om(f_*\omega_{Y/V}(S_Y-f^*(K_{X/V}+S+B)),j_*\sO_{S^N})$ is torsion-free as $\sO_S$-module, it is enough to show that the above diagram commutes at every generic point of $S$.
Thus, we may assume that $(X,S)$ is a simple normal crossing pair and $B=0$.
By the construction of the residue map $\omega_{X/V}(S) \to \omega_{S/V}$ and the trace map, we have the commutative diagram
\[
\xymatrix{
f_*\omega_{Y/V}(S_Y) \ar[r] \ar[d] & \omega_{X/V}(S) \ar[d] \\
f_{S,*} \omega_{S_Y/V}  \ar[r] & \omega_{S/V}.
}
\]
Applying $\otimes \sO_X(-(K_{X/V}+S))$, we have the assertion.
\end{proof}

\begin{prop}\label{prop:adjunction}\textup{(Adjunction for global $T$-regularity)}
Let $V$ be a scheme satisfying Assumption \ref{assump:cdvr}.
Let $f \colon X \to Z$ be a projective surjective morphism from a normal $V$-variety $X$ to an affine $V$-variety $Z$ and $(X,S+B)$ a purely globally $T$-regular  pair with $\lfloor S+B \rfloor =S$.
Then $(S^N,B_S)$ is globally $T$-regular, where $B_S=\mathrm{Diff}_{S^N}(B)$ is the different of $(X,S+B)$.
\end{prop}

\begin{proof}
We take an alteration $g \colon T \to S^N$.
Combining Proposition \ref{prop:extension of alterations of closed variety} and Lemma \ref{lem:good alteration}, there exists an alteration $f \colon Y \to X$ and a strict transform $S_Y$ of $S$ as in Lemma \ref{lem:good alteration} and $f_S \colon S_Y \to S^N$ factors through $g$.
Then we have the commutative diagram
\[
\xymatrix{
H^0(\omega_{Y/V}(S_Y-f^*(K_{X/V}+S+B))) \ar@{->>}[r] \ar[d] & H^0(\sO_X) \ar[d] \\
H^0(\omega_{S_Y/V}(-f_S^*(K_{S^N/V}+B_S))) \ar[r] & H^0(\sO_{S^N}).
}
\]
Thus, the image of the bottom map contains $1 \in H^0(\sO_{S^N})$, and in particular, it is surjective.
\end{proof}

\begin{prop}\label{prop:inversion of adjunction globally $T$-regular}\textup{(Inversion of Adjunction for global $T$-regularity)}
Let $V$ be a scheme satisfying Assumption \ref{assump:cdvr}.
Let $f \colon X \to Z$ be a projective surjective morphism from a normal $V$-variety $X$ to an affine $V$-variety $Z$ and $(X,S+\Delta)$  a log pair such that $\lfloor S+\Delta \rfloor=S$ is reduced and $S$ has no common component with $\Delta$.
Let $X \to Z' \to Z$ be the Stein factorization and we denote the induced morphisms by $f' \colon X \to Z'$ and $\phi \colon Z' \to Z$.
Let $z \in Z$ be a point with $\phi^{-1}(z) \subset f'(S)$.
We assume that $-(K_{X/V}+S+\Delta)$ is semiample and $f$-big.
If $(S^N,D)$ is globally $T$-regular over $z$, then $(X,S+\Delta)$ is purely globally $T$-regular over $z$ and $S$ is normal, where $S^N$ is the normalization of $S$ and $D=\mathrm{Diff}_{S^N}(\Delta)$ is the different of the pair $(X,S+\Delta)$.

\end{prop}

\begin{proof}
We denote $H:=-(K_{X/V}+S+\Delta)$ and $H_S:=-(K_{S^N/V}+D)$.
We take base changes via $\Spec{\sO_{Z,z}} \to Z$ and we use the same notations by abuse of notations.
%We may replace $Z$ with $\Spec \sO_{Z,z}$.
By the assumption $\phi^{-1}(z) \subset f'(S)$, the ideal $H^0(\sO_X(-S))$ of $H^0(\sO_X)=H^0(\sO_{Z'})$ is contained in all maximal ideals of $H^0(\sO_X)$.
We take an alteration $g \colon Y \to X$ and a strict transform $S_Y$ of $S$ as in Lemma \ref{lem:good alteration}.
In order to show that $(X,S+\Delta)$ is purely globally $T$-regular, it is enough to show that the image $I_g \subset H^0(\sO_X)$ of the map
\[
H^0(\omega_{Y/V}(S_Y+g^*H)) \to H^0(\sO_X)
\]
is $H^0(\sO_X)$.
Since $H^0(\sO_X(-S))$ is contained in all maximal ideals, it is enough to show $\psi(I_g) =H^0(\sO_{S^N})$, where $\psi \colon H^0(\sO_X) \to H^0(\sO_{S^N})$.
We take an element $\alpha \in H^0(\sO_{S^N})$.
Since $g^*H$ is semiample and big over $Z$, there exists an alteration $h \colon W \to Y$ and a strict transform $S_W$ of $S$ as in Lemma \ref{lem:good alteration} and the trace map
\[
H^1(\omega_{W/V}(h^*g^*H)) \to H^1(\omega_{Y/V}(g^*H))
\]
is zero by Corollary \ref{cor:zero map} and Remark \ref{rmk:zero map}.
Since $(S^N,D)$ is globally $T$-regular, there exists a section 
\[
\alpha_W \in H^0(\omega_{S_W/V}(h_S^*g_S^*H_S))
\]
mapped to $\alpha \in H^0(\sO_{S^N})$ by the trace map.
Let $\alpha_Y \in H^0(\omega_{S_Y/V}(g_S^*H_S))$ be the image of $\alpha_W$ via the trace map.
By the exact sequence
\[
0 \to \omega_{Y/V}(g^*H) \to \omega_{Y/V}(S_Y+g^*H) \to \omega_{S_Y/V}(g_S^*H_S) \to 0,
\]
we have the exact sequence
\[
H^0(\omega_{Y/V}(S_Y+g^*H)) \to H^0(\omega_{S_Y/V}(g_S^*H_S)) \to H^1(\omega_{Y/V}(g^*H)).
\]
Thus we have the following commutative diagram
\[
\xymatrix{
H^0(\omega_{S_W/V}(h_S^*g_S^*H_S))) \ar[r] \ar[d] & H^1(\omega_{W/V}(h^*g^*H)) \ar[d]^{0-\text{map}} \\
H^0(\omega_{S_Y/V}(g_S^*H_S)) \ar[r] & H^1(\omega_{Y/V}(g^*H))
}
\]
by Lemma \ref{lem:good alteration}.
Thus the image of $\alpha_Y$ via the connection map is zero, so  $\alpha_Y$ extends to a section $\gamma \in H^0(\omega_{Y/V}(S_Y+ g^*H )) $ and its image is $\alpha$ in $H^0(\sO_{S^N})$.
Thus we have $\alpha \in \psi(I_g)$, and the equation $\psi(I_g)=H^0(\sO_{S^N})$ holds.

Next, we prove the normality of $S$.
We take an open affine covering $\{U_i\}$ of $X$.
By Proposition \ref{prop:localization globally $T$-regular}, each $(U_i|_{S^N},D|_{U_i|_{S^N}})$ is globally $T$-regular. 
This is a local problem, we may assume that $X$ is affine by Proposition \ref{prop:localization globally $T$-regular}.
By the above argument, we have $\psi(\sO_X)=\sO_{S^N}$, in particular, $S$ is normal.
\end{proof}

\begin{cor}\label{cor:normality of globally $T$-regular}
Let $V$ be a scheme satisfying Assumption \ref{assump:cdvr}.
Let $X$ be a normal affine $V$-variety and $(X,S+\Delta)$ be a log pair such that $\lfloor S+\Delta \rfloor=S$ is reduced.
Then $(X,S+\Delta)$ is purely $T$-regular if and only if $(S^N,\mathrm{Diff}_{S^N}(\Delta))$ is $T$-regular.
Furthermore, in both cases, $S$ is normal, and in particular, $S$ is locally irreducible.
\end{cor}

\begin{proof}
It follows from Proposition \ref{prop:adjunction} and Proposition \ref{prop:inversion of adjunction globally $T$-regular} for the case $f=\id$.
\end{proof}

\begin{cor}\label{cor:T-regularity of simple normal crossing pair}
Let $V$ be a scheme satisfying Assumption \ref{assump:cdvr}.
Let $(X,\Delta)$ be a simple normal crossing pair with $\lfloor \Delta \rfloor=0$, where $X$ is an affine $V$-variety.
Then $(X,\Delta)$ is globally $T$-regular.
\end{cor}

\begin{proof}
We prove Corollary \ref{cor:T-regularity of simple normal crossing pair} by the induction on $d:=\dim X$.
We take an alteration $\pi \colon Y \to X$ from a normal $V$-variety $Y$.
It is enough to show that the trace map is surjective at each closed point $x \in X$. 
First, we consider the case where $x$ is not contained in $\Supp(\Delta)$.
By \cite[Theorem 1.2]{bhatt18}, 
\[
\sO_X \to R\pi_*\sO_Y
\]
splits.
By the Grothendieck duality, the map $\pi_*\omega_{Y/V} \to \omega_{X/V}$ is surjective.
Since $X$ is Gorenstein, the trace map
\[
\pi_{*}
\omega_{Y/V}(-\pi^*K_{X/V}) \to \sO_X
\]
is also surjective.
Next, we assume that $x$ is contained in a component $S$ of $\Supp(\Delta)$.
We take a positive rational number $a$ with $\ord_S(\Delta+aS)=1$.
By Proposition \ref{prop:purely globally $T$-regular to globally $T$-regular}, it is enough to show that $(X,\Delta+aS)$ is purely $T$-regular at $x$.
By Corollary \ref{cor:normality of globally $T$-regular}, it is enough to show that $(S,(\Delta-(1-a)S)|_S)$ is $T$-regular at $x$.
Since this pair is a simple normal crossing pair, this pair is $T$-regular by the induction hypothesis on $d$.
In conclusion, $(X,\Delta)$ is $T$-regular.
\end{proof}

\begin{rmk}
By the proof of \cite[Theorem 6.21]{ma-schwede18}, if a log pair $(X,\Delta)$ is BCM-regular at $x$ after completion, then $(X,\Delta)$ is $T$-regular at $x$.
Thus, Corollary \ref{cor:T-regularity of simple normal crossing pair} also follows from \cite[Theorem 4.1]{mstww}.
\end{rmk}

\begin{prop}\label{prop:more inversion}
Let $V$ be a scheme satisfying Assumption \ref{assump:cdvr}.
Let $f \colon X \to Z$ be a projective surjective morphism from a normal $V$-variety $X$ to an affine $V$-variety $Z$ and $(X,S+\Delta)$ a log pair such that $\lfloor S+\Delta \rfloor=S$ is a reduced divisor and $S$ has no common component with $\Delta$.
Let $L$ be a $\Q$-Cartier Weil divisor such that $L-(K_{X/V}+S+\Delta)$ is semiample and $f$-big and $L$ is Cartier at all codimension two points of $X$ contained in $S$.
Let $g \colon Y \to X$ be an alteration from a normal $V$-variety $Y$ and $S_Y$ a strict transform of $S$.
Then we have 
\[
T^0(S^N, D ;L|_S) \subset I_g|_{S^N},
\]
where $D:=\mathrm{Diff}_{S^N}(\Delta)$ is the different, $I_g$ is the image of the trace map
\[
I_g:=\mathrm{Im}(H^0(\omega_{Y/V}(S_Y+\lceil g^*(L-(K_{X/V}+S+\Delta)) \rceil)) \to H^0(\sO_X(L)))
\]
and the right hand side is the image of $I_g$ via the natural map 
\[
H^0(\sO_X(L)) \to H^0(\sO_{S^N}(L|_{S^N})).
\]

\end{prop}

\begin{proof}
We denote $H:=L-(K_{X/V}+S+\Delta)$ and $H_S:=L|_{S^N}-(K_{S^N/V}+D)$.
We note that $\mathrm{Diff}_{S^N}(\Delta-L)=D-L|_{S^N}$ since $L$ is Cartier at all codimension two points of $X$ contained in $S$.
We may assume that $(Y,S_Y)$ is as in Lemma \ref{lem:good alteration}.
By the proof of Proposition \ref{prop:inversion of adjunction globally $T$-regular}, we can construct an alteration $h \colon W \to Y$ and a strict transform  $S_W$ of $S$ as in Lemma \ref{lem:good alteration} such that the trace map on the first cohomology is zero.
We obtain the commutative diagram
\[
\xymatrix{
H^0(\omega_{W/V}(S_W+h^*g^*H)) \ar[r] \ar[d]& H^0(\omega_{S_W/V}(h_S^*g_S^*H_S)) \ar[r] \ar[d] & H^1(\omega_{W/V}(h^*g^*H))) \ar[d]^{0-\text{map}} \\
H^0(\omega_{Y/V}(S_Y+g^*H)) \ar[r] \ar[d] & H^0(\omega_{S_Y/V}(g_S^*H_S)) \ar[r] \ar[d] & H^1(\omega_{Y/V}(g^*H)) \\
H^0(\sO_X(L)) \ar[r]  & H^0(\sO_{S^N}(L|_{S^N})). &
}
\]
Thus the proof is same as the proof of Proposition \ref{prop:inversion of adjunction globally $T$-regular}.
\end{proof}

\begin{prop}\label{prop:gfr vs gs}
Let $k$ be an $F$-finite field of positive characteristic.
Let $f \colon X \to Z$ be a projective surjective morphism from a normal $k$-variety $X$ to an affine $k$-variety $Z$.
Let $(X,B)$ be a log pair.
If $(X,B)$ is globally $F$-regular, then it is globally $T$-regular.
%Let $\pi \colon Y \to X$ be an alteration from a normal variety $Y$.
%If the trace map
%\[
%H^0(\omega_Y(\lceil -\pi^*(K_X+B) \rceil) \to H^0(\sO_X)
%\]
%is non-zero, then it is surjective.
%In particular, if we further assume that the generic fiber of $(X,B) \to Z$ is  globally $T$-regular, then  $(X,B)$ is  globally $T$-regular.
\end{prop}

\begin{proof}
Let $H$ be an ample Cartier divisor on $X$ such that $\sO_X(K_X+B+H)$ is globally generated.
We take an effective divisor $B'$ which is linearly equivalent to $K_X+B+H$.
Then for divisible enough $e$, $(X,B+\frac{1}{p^e-1}B')$ is globally $F$-regular, the Cartier index of
\[
(p^e-1)(K_X+B+\frac{1}{p^e-1}B') \sim p^e(K_X+B)+H
\]
is prime to $p$, and $\lceil -\pi^*(K_X+B+\frac{1}{p^e-1}B') \rceil=\lceil -\pi^*(K_X+B) \rceil$.
Thus, we may assume that the Cartier index of $K_X+B$ is prime to $p$.
By \cite[Lemma 6.14]{bmpstww} and Proposition \ref{prop:corresponding to finite cover in positive chara}, the generic fiber of $f$ is globally $T$-regular.
In particular, for a given alteration $\pi \colon Y \to X$, the trace map
\[
H^0(\omega_Y(\lceil -\pi^*(K_X+B) \rceil) \to H^0(\sO_X)
\]
is non-zero.
We take an element $\alpha \in H^0(\omega_Y(\lceil -\pi^*(K_X+B) \rceil))$ such that $\beta:=\Tr(\alpha) \in H^{0}(\sO_{X})$ is non-zero.
Since $(X,B)$ is globally $F$-regular, there exists a positive integer $e$ such that $(p^e-1)(K_X+B)$ is Cartier and the natural map
\[
\sO_X \to F^e_*\sO_X(\lfloor (p^e-1)B \rfloor + \mathrm{div}(\beta))
\]
splits.
Let $\Gamma:=(1-p^e)(K_X+B)$.
We take a section $\gamma \in H^0(\sO_X(\Gamma))$ corresponding to the splitting.
Then the trace map
\[
H^0(\sO_X(\Gamma) \to H^0(\sO_X)
\]
maps $\gamma \beta$ to $1$.
We regard $\alpha \gamma$ as an element of
\[
H^0(\omega_Y(\lceil -p^e\pi^*(K_X+B) \rceil)) \simeq H^0(\pi_*\omega_Y(\lceil -\pi^*(K_X+B) \rceil)\otimes \sO_X(\Gamma)).
\]
Then the composition of morphisms
\[
\xymatrix{
H^0(\omega_Y(\lceil -p^e\pi^*(K_X+B) \rceil)) \ar[r] & H^0(\sO_X( \Gamma )) \ar[r] & H^0(\sO_X)
}
\]
maps $\alpha \gamma$  to $1$, where the first map is induced by
\[
\mathrm{Tr} \otimes \sO_X(\Gamma) \colon \pi_*\omega_Y(\lceil -\pi^*(K_X+B) \rceil)\otimes \sO_X(\Gamma) \to \sO_X(\Gamma).
\]
Thus, the map
\[
H^0(\omega_Y(\lceil -\pi^*(K_X+B) \rceil) \to H^0(\sO_X)
\]
is surjective.
\end{proof}

\subsection{Restriction theorem}
The goal of this subsection is to prove the existence of three-dimensional pl-flips with ample divisor in the boundary (Corollary \ref{cor:existence of pl-flip threefold}).
First, we prove the existence by assuming the global $T$-regularity of the boundary even in the higher-dimensional case (Theorem \ref{thm:pl-filp in ample divisor}).
It follows from the restriction theorem (Proposition \ref{prop:extension result}) and Shokurov's reduction to pl-flips.
Next, we show this condition in the three-dimensional case.

\begin{lem}\label{lem:pertubation lemma}\textup{(cf.\,\cite[Lemma 3.2]{hacon-witaszek20})}
Let $V$ be a scheme satisfying Assumption \ref{assump:cdvr}.
Let $f \colon X \to Z$ be a projective birational morphism from a normal $V$-variety to an affine $V$-variety.
Let $(X,B)$ be a log pair which is globally $T$-regular over a point $z$ of $Z$.
Let $L$ be a Weil $\Q$-Cartier divisor on $X$ and $\Gamma$ an effective $\Q$-Cartier $\Q$-Weil divisor on $X$.
If $g \in H^0(X,\sO_X(L))$ corresponds to a divisor $G \in |L|$ such that $G \geq \Gamma$,
then $g$ is contained in $T^0(X,B+\Gamma;L)$ after localizing at $z$.
\end{lem}

\begin{proof}
We take base changes via $\Spec{\sO_{Z,z}} \to Z$ and we use the same notations by abuse of notations.
%We may replace $Z$ with $\Spec\sO_{Z,z}$.
It is enough to consider alterations $h \colon Y \to X$ from a normal $V$-variety such that $h^*(K_{X/V}+B)$, $h^*L$, $h^*G$ and $h^*\Gamma$ are Cartier.
We consider the commutative diagram
\[
\xymatrix{
H^0(\omega_{Y/V}(h^*(L-K_{X/V}-B-\Gamma)))  \ar[r] & H^0(\sO_X(L)) \\
H^0(\omega_{Y/V}(h^*(L-K_{X/V}-B-G))) \ar[r] \ar@{^{(}->}[u] & H^0(\sO_X(L-G)) \ar@{^{(}->}[u] \\
H^0(\omega_Y(-h^*(K_{X/V}+B))) \ar@{->>}[r] \ar[u]^{\simeq}_{\cdot g} & H^0(\sO_X) \ar[u]^{\simeq}_{\cdot g},
}
\]
where the horizontal maps are trace maps and the surjectivity of the bottom map  follows from the global $T$-regularity of $(X,B)$.
Thus, we have a section of $H^0(\omega_{Y/V}(h^*(L-K_{X/V}-B-\Gamma))) $ mapped to $g$.
\end{proof}

\begin{prop}\label{prop:extension result}\textup{(cf.\,\cite[Proposition 3.1]{hacon-witaszek20})}
Let $V$ be a scheme satisfying Assumption \ref{assump:cdvr}.
Let $f \colon X \to Z$ be a projective morphism from a normal $V$-variety $X$ to an affine $V$-variety $Z$.
Let $(X,S+A+B)$ be a dlt pair such that $S$ and $A$ are  $\Q$-Cartier Weil divisors.
Assume that  $A$ is ample  and $\lfloor B \rfloor=0$.
Let $z \in Z$ be a point.
If $S$ is normal and $(S,(1-\varepsilon)A_S+B_S)$ is globally $T$-regular over $z$ for all $0<\varepsilon<1$,
then for every $k \geq 1$ such that $k(K_{X/V}+S+A+B)$ is Cartier, we have
\[
|k(K_{X/V}+S+A+B)|_S=|k(K_{S/V}+A_S+B_S)|
\]
after localization at $z$, where $B_S:=\mathrm{Diff}_{S}(B)$ and $A_S:=A|_S$.
\end{prop}

\begin{proof}
We take base changes via $\Spec{\sO_{Z,z}} \to Z$ and we use the same notations by abuse of notations.
%We may replace $Z$ with $\Spec\sO_{Z,z}$.
%First, we note that $X$ is globally $T$-regular by Proposition \ref{prop:inversion of adjunction globally $T$-regular}, thus all divisorial sheaf on $X$ is maximal Cohen-Macaulay by Proposition \ref{prop:first properties of glovally splinter} $(3)$.
We note that since $(X,S+A+B)$ is dlt and $S$ and $A$ are $\Q$-Cartier, $(X,S+A+B)$ is a simple normal crossing pair around $S \cap A$, and in particular, $A$ is Cartier at all codimension two points on $X$ contained in $S$.
We take $F \in |k(K_{S/V}+A_{S}+B_{S})|$. 
By the descending induction, we prove that there exists divisors $G_m \in |k(K_{X/V}+S+A+B)+mA|$ such that $G_m|_S=F+mA_S$.
First, since $A$ is ample, such $G_m$ exists for large enough $m$ by Serre vanishing.
We assume that such $G_{m+1}$ exists.
We set
\begin{eqnarray*}
L &:=& k(K_{X/V}+S+A+B)+mA \\
&=& K_{X/V}+S+B+(k-1)(K_{X/V}+S+A+B)+(m+1)A \\
&\sim_{\Q}& K_{X/V}+S+B+\frac{k-1}{k}G_{m+1}+\frac{m+1}{k}A.
\end{eqnarray*}
We write $H:=L-(K_{X/V}+S+B+\frac{k-1}{k}G_{m+1})$, then it is ample.
Since $(S,\frac{k-1}{k}A_S+B_S)$ is globally $T$-regular and 
\[
F+mA_S \geq \frac{k-1}{k}(F+mA_S)=\frac{k-1}{k}G_{m+1}|_S-\frac{k-1}{k}A_S,
\]
a section $g \in H^0(L|_S)$ corresponding to $F+mA_S$ is contained in 
\[
T^0(S,B_S+\frac{k-1}{k}G_{m+1}|_S; L|_S)
\]
by Lemma \ref{lem:pertubation lemma}.
By Proposition \ref{prop:more inversion}, $g$ is contained in the image of $H^0(\sO_X(L))$.
%By proposition \ref{cor:zero map}, there exists an alteration $h \colon Y \to X$ and a prime divisor $S_Y$ on $Y$ such that $(Y,S_Y)$ is simple normal crossing, $h^*H$ is Cartier and the map
%\[
%H^1(\omega_Y(h^*H)) \to H^1(\omega_X(\lceil H \rceil))
%\]
%is zero.
%Since $\omega_X(\lceil H \rceil)$ is maximal Cohen-Macaulay, we have the commutative diagram of exact sequences
%\[
%\xymatrix{
%0 \ar[r] & \omega_X(\lceil H \rceil) \ar[r] & \omega_X(S+\lceil H \rceil) \ar[r]  & \omega_S(\lceil H \rceil|_S) \ar[r]  & 0 \\
%0 \ar[r] & \omega_Y(h^*H) \ar[r] \ar[u] & \omega_Y(S_Y+h^*H) \ar[r] \ar[u] & \omega_{S_Y}(h^*H_S) \ar[r] \ar[u] & 0,
%}
%\]
%where $H_S:=L|_S-(K_S+B_S+\frac{k-1}{k}G_{m+1})$.
%Furthermore, we obtain the commutative diagram of exact sequences
%\[
%\xymatrix{
%H^0(\omega_X(S+\lceil H \rceil)) \ar[r]  & H^0(\omega_S(\lceil H \rceil|_S)) \ar[r]  & H^1(\omega_X(\lceil H \rceil))   \\
%H^0(\omega_Y(S_Y+h^*H)) \ar[r] \ar[u] & H^0(\omega_{S_Y}(h^*H_S)) \ar[r] \ar[u] & H^1(\omega_Y(h^*H)) \ar[u]_{0-\mathrm{map}}.
%}
%\]
%We note that the upper middle map is decomposable as follows;
%\begin{eqnarray*}
%H^0(\omega_Y(h^*H_S))&=& H^0(\omega_Y(h^*(L|_S-(K_S+B_S+\frac{k-1}{k}G_{m+1})))) \\
%&\to& H^0(\sO_S(L|_S)) =H^0(\omega_S(H|_S)) \\
%&\to& H^0(\omega_S(\lceil H \rceil|_S)).
%\end{eqnarray*}
%In particular, $g$ is contained in the image of the above map.
%Thus, $g$ extends to the global section of $\omega_X(S+\lceil H \rceil)$.
%Because of $K_X+S+\lceil H \rceil \leq L$, we take a divisor $G_m \in |L|$ such that $G_m|_S=F+mA_S$.
\end{proof}

The following theorem is the existence of pl-flips with ample divisor in the boundary in the special setting.
The proof is an analog of the proof of \cite[Theorem 1.3]{hacon-witaszek20}.

\begin{thm}\label{thm:pl-filp in ample divisor}\textup{(cf.\,\cite[Theorem 1.3]{hacon-witaszek20})}
Let $V$ be a scheme satisfying Assumption \ref{assump:cdvr}.
Let $(X,S+A+B)$ be a dlt pair such that $S$ is an anti-ample $\Q$-Cartier Weil divisor and $A$ is an ample $\Q$-Cartier Weil divisor.
Let $f \colon X \to Z$ be a $(K_{X/V}+S+A+B)$-flipping contraction with $\rho(X/Z)=1$ to an affine $V$-variety.
Furthermore, we assume that $(S^N,(1-\varepsilon)A_S+B_S)$ is globally $T$-regular for all $0<\varepsilon<1$  over all points of $f(\Exc(f))$ and $R(K_{S^N/V}+A_S+B_S)$ is finitely generated, where  $B_S:=\mathrm{Diff}_{S^N}(B)$ and $A_S:=A|_S$.
Then the flip of $f$ exists.
\end{thm}

\begin{proof}
Take a point $z \in Z$ contained in $f(\Exc(f))$.
Since $A$ is ample, we have $z \in f(A)$.
We take base changes via $\Spec{\sO_{Z,z}} \to Z$ and we use the same notations by abuse of notations.
%We may replace $Z$ with $\Spec\sO_{Z,z}$
We may assume that $\lfloor B \rfloor=0$.
By Proposition \ref{prop:inversion of adjunction globally $T$-regular}, the scheme $S$ is normal.
By Proposition \ref{prop:extension result}, the restriction algebra
\[
R_S(k(K_{X/V}+S+A+B)):=\mathrm{Im}(R(k(K_{X/V}+S+A+B)) \to R(k(K_{S/V}+A_S+B_S)))
\]
coincides with $R(k(K_{S/V}+A_S+B_S))$ for some positive integer $k$, and in particular, $R_S(K_{X/V}+S+A+B)$ is finitely generated.
By Shokurov's reduction to finite generation (see  \cite[Lemma 2.3.6]{corti}), the flip of $f$ exists.
\end{proof}

\begin{rmk}
If $(S^N,(1-\varepsilon)A_S+B_S)$ is globally $F$-regular over $Z$, then it is globally $T$-regular over $Z$ by Proposition \ref{prop:gfr vs gs}.
Thus, applying Theorem \ref{thm:pl-filp in ample divisor} for the case $X=X_s$, we obtain \cite[Theorem 1.3]{hacon-witaszek20}.
\end{rmk}

In order to use Theorem \ref{thm:pl-filp in ample divisor} for threefolds, we will show the pure global $T$-regularity of $(S^N,A_S+B_S)$.

\begin{lem}\label{lem:bcm regular and globally $T$-regular surface}\textup{(cf.\,\cite[Lemma 3.3]{hacon-witaszek19})}
Let $V$ be a scheme satisfying Assumption \ref{assump:cdvr}.
We assume that the residue field of $V$ is infinite.
Let $f \colon S \to T$ be a projective birational morphism from a normal $V$-surface $S$ to an affine $V$-surface $T$.
Let $(S,C+B)$ be a plt pair with $\lfloor C+B \rfloor=C$.
Assume that $-(K_{S/V}+C+B)$  and $C$ are ample.
Further assume that $f$ has connected fibers.
Then $(S,C+B)$ is purely globally $T$-regular over all points of $f(\Exc(f))$.
\end{lem}

\begin{proof}
Since $f$ has connected fibers, the Stein factorization of $f$ induces a homeomorphism $\phi \colon T' \to T$.
Thus we have $\phi^{-1}(f(\Exc(f))=f'(\Exc(f'))$, so we may assume that $T$ is normal and $f_*\sO_X=\sO_T$ by replacing $T$ into $T'$, where $f' \colon X \to T'$ is the induced morphism.
We take a point $t \in f(\Exc(f))$.
First, we prove that $C$ is irreducible after shrinking $T$ around $t$.
By shrinking $T$ around $t$, the image of every irreducible component of $C$ contains $t$.
Since $S$ is surface and $(S,C+B)$ is plt, then $C$ is locally irreducible.
By \cite[Theorem 5.2]{tanaka18}, the intersection $C \cap f^{-1}(t)$ is connected, and in particular, the scheme $C$ is irreducible.
Furthermore, as $C$ is ample, $C$ is not an exceptional divisor of $f$.
By adjunction, $(C,\mathrm{Diff}_C(B+C))$ is normal klt one-dimensional pair, in particular, it is simple normal crossing.
Since $C$ is not exceptional and $T$ is affine, $C$ is also affine.
Therefore, $(C,\mathrm{Diff}_C(B+C))$ is $T$-regular by Corollary \ref{cor:T-regularity of simple normal crossing pair}.
Since $-(K_{S/V}+C+B)$ is ample, $(S,B+C)$ is purely globally $T$-regular by Proposition \ref{prop:inversion of adjunction globally $T$-regular}.
\end{proof}

\begin{lem}\label{lem:pl-flip good condition}
Let $V$ be a scheme satisfying Assumption \ref{assump:cdvr}.
Assume that the residue field of $V$ is infinite.
Let $f \colon X \to Z$ be a small projective birational morphism from a  normal $V$-variety $X$ of dimension three to an affine $V$-variety $Z$.
Let $(X,S+A+B)$ be a dlt  pair such that $-(K_{X/V}+S+A+B)$ is ample, $S$ and $A$ are locally irreducible $\Q$-Cartier Weil divisors and $\lfloor B \rfloor=0$.
Assume that $-S$  and $A$ are ample.
Then $(S^N,\mathrm{Diff}_{S^N}(A+B))$ is purely globally $T$-regular over all points of $f(\Exc(f))$.
In particular, $S$ is normal over a neighborhood of $f(\Exc(f))$.
\end{lem}

\begin{proof}
Since $K_{X/V}+S+B$ is $\Q$-Cartier, $\mathrm{Diff}_{S^N}(A+B)=D+A|_{S^N}$, where $D:=\mathrm{Diff}_{S^N}(B)$.
We note that $A$ is Cartier on the codimension two points of $X$ contained in $S$.
Since $f$ is small, $f|_{S^N} \colon S^N \to T $ is birational, where $T:=f(S)$.
Since $-S$ is ample, all exceptional curves of $f$ are contained in $S$, thus $f|_S \colon S \to T$ has connected fibers.
Since $S^N \to S$ is a universal homeomorphism by Lemma \ref{lem:topologically normal}, $f|_{S^N}$ also has connected fibers.
By Proposition \ref{prop:dlt adjunction to plt}, the pair $(S^N,D+A|_{S^N})$ is plt.
By Lemma \ref{lem:bcm regular and globally $T$-regular surface}, $(S^N,D+A|_{S^N})$ is purely globally $T$-regular over all points of $f(\Exc(f))$.
In particular, $S$ is normal over a neighborhood of $f(\Exc(f))$ by Proposition \ref{prop:inversion of adjunction globally $T$-regular}.
\end{proof}

\begin{cor}\label{cor:existence of pl-flip threefold}\textup{(cf.\,\cite[Proposition 3.4]{hacon-witaszek19})}
%Let $V$ be the spectrum of a divisorial valuation ring of characteristic $(0,p)$ with  $F$-finite residue field.
Let $V$ be an excellent Dedekind scheme.
Let $(X,S+A+B)$ be a three-dimensional dlt pair over $V$.
Let $f \colon X \to Z$ be a $(K_{X/V}+S+A+B)$-flipping contraction with $\rho(X/Z)=1$.
Assume that $S$ and $A$ are locally irreducible Weil divisors such that $-S$ and $A$ are ample $\Q$-Cartier.
Then the flip of $f$ exists.
\end{cor}

\begin{proof}
We may assume that $V$ is the spectrum of an excellent discrete valuation ring.
We may assume $\lfloor B \rfloor=0$.
Since the existence of flip is a local problem on $Z$, we may assume that $Z$ is affine.
By Shokurov's reduction to pl-flip (see  \cite[Lemma 2.3.6]{corti}), it is enough to show that
$R_S(k(K_X+S+A+B))$ is finitely generated.
This statement can be reduced to the case where $V$ is complete and the residue field is infinite taking a strict henselization and completion.
By Lemma \ref{lem:base change}, the assumption is preserved except for the condition that the relative Picard rank is one.
By Lemma \ref{lem:pl-flip good condition}, $S$ is normal and $(S,A_S+B_S)$ is purely globally $T$-regular over all points of $f(\Exc(f))$, where $B_S=\mathrm{Diff}_{S}(B)$ and $A_S=A|_{S}$.
%By the proof of Theorem \ref{thm:pl-filp in ample divisor}, it is enough to show that $R(K_S+A_S+B_S)$ is finitely generated.
By Theorem \ref{thm:pl-filp in ample divisor}, it is enough to show that $R(K_{S/V}+A_S+B_S)$ is finitely generated.
We take an effective divisor $A'$ on $S$ with $A_S \sim_{\Q} A'$ such that $A_S$ and $A'$ have no common component.
Since $(S,A_S+B_S)$ is plt, $(S,(1-\varepsilon)A_S+B_S+\varepsilon A')$ is klt for small enough positive rational number $\varepsilon$.
By \cite[Theorem 1.1, Theorem 4.2, Corollary 4.11]{tanaka18},
the canonical ring
$R(K_{S/V}+(1-\varepsilon)A_S+B_S+\varepsilon A')$ is finitely generated.
Since 
\[
K_{S/V}+(1-\varepsilon)A_S+B_S+\varepsilon A'\sim_{\Q} K_{S/V}+A_S+B_S,
\]
$R(K_{S/V}+A_S+B_S)$ is also finitely generated.
\end{proof}

\begin{rmk}
The existence of necessary flips for \cite[Theorem 6]{kollar20} follows from Corollary \ref{cor:existence of pl-flip threefold} over an excellent Dedekind scheme.
\end{rmk}

\section{Proof of Theorem \ref{thm:birational relative mmp} and its applications}
The goal of this section is to prove Theorem \ref{thm:birational relative mmp} and it's applications.
Proposition \ref{prop:existence of flip trivial case} is one of the applications and it will be used to prove Theorem \ref{thm:semi-stable mmp}.
To prove these theorems, we establish the cone theorem for pseudo-effective pairs (Proposition \ref{prop:cone theorem}), by following the method given by \cite{keel} \cite{tanaka18a}.
We also prove the cone theorem for more general settings (Proposition \ref{prop:special fiber cone thm}) by using the method given by \cite{kawamata94}.
If every relative curve is contained in the special fiber, then the cone theorem is easily reduced to the case of surfaces, but in the relative setting, relative curves contained in the generic fiber may exist.  
Therefore, we should treat such cases carefully.
%In the relative setting, we should treat relative curves which is contained in the generic fiber (see the remark ).
%We note that there may be relative curves which is contained in the generic fiber in 

\begin{prop}\label{prop:contraction anti-ample case}
Let $V$ be an excellent Dedekind scheme.
Let $(X,S+B)$ be a three-dimensional dlt  pair over $V$ such that $S$ is a $\Q$-Cartier Weil divisor.
Let $\rho \colon X \to U$ be a projective morphism over $V$.
Let $\Sigma$ be a $(K_{X/V}+S+B)$-negative extremal ray contracted by $\rho$.
Let $L$ be a $\rho$-nef Cartier divisor on $X$ with $L^{\perp}=\R[\Sigma]$.
Assume that $S$ is a prime divisor and $S \cdot \Sigma < 0$.
Then $L$ is semiample over $U$.
\end{prop}

\begin{proof}
It follows from a similar argument to the argument in the proof of [HW, Proposition 4.4] by replacing [HW, Lemma 2.1] with Lemma \ref{lem:topologically normal} and using \cite[Theorem 1.2]{witaszek20}.
\end{proof}

\begin{prop}
\label{prop:cone theorem}
Let $V$ be an excellent Dedekind scheme.
Let $\pi \colon X \rightarrow U$ be a projective $V$-morphism from a normal $\Q$-factorial quasi-projective $V$-threefold $X$ to a quasi-projective  $V$-variety $U$.
Let $B$ be an effective $\R$-Weil divisor on $X$ satisfying the following.
\begin{itemize}
    \item %$0 \leq B \leq 1$,
    all coefficients $c$ of $B$ satisfy
    $0 \leq c \leq 1$, and
    \item $K_{X/V} + B$ is pseudo-effective.
\end{itemize}
Let $A$ be a $\pi$-ample $\R$-Cartier $\R$-divisor on $X$.
Then there exist finitely many $\pi$-relative curves $C_{1}, \ldots, C_{r}$ on $X$ such that
\[
\NE (X/U) = \NE (X/U)_{K_{X/V}+B+A \geq 0} + \sum_{i=1}^{i=r} \R_{\geq 0} [C_{i}].
\]
\end{prop}
\begin{proof}
The assertion is proved by the same method as in \cite[Theorem 7.6]{tanaka18a}.
Here, we use \cite[Theorem 2.14]{tanaka18} after the reduction to the case of surfaces.
\end{proof}

In the proof of Proposition \ref{prop:existence of flip trivial case}, we run an MMP with scaling.
In order to do this, we prepare the following corollary.

\begin{cor}\label{cor:mmp with scaling}
Let $V$ be an excellent Dedekind scheme.
Let $(X,\Delta)$ be a dlt $\Q$-factorial pair over $V$ satisfying that $\lfloor \Delta \rfloor = X_{s}$ as sets.
Let $\pi \colon X \rightarrow U$ be a projective birational morphism over $V$ from $X$ to a quasi-projective $V$-variety $U$.
Let $H$ be a $\Q$-Cartier $\Q$-Weil divisor such that $K_{X/V} + \Delta + H$ is $\pi$-nef.
We put 
\[
\lambda_{H} := \inf\{ \lambda \in \R_{\geq 0} \mid K_{X/V} + \Delta + \lambda H \textup{ is } \pi \textup{-nef} \}.
\]
Then there exists a $(K_{X/V}+ \Delta)$-negative extremal ray $R \subset \NE(X/U)$ satisfying that 
\[
(K_{X/V} + \Delta + \lambda_{H} H) \cdot R = 0.
\]
\end{cor}
\begin{proof}
Take a rational number $a \in \Q_{>0}$ with $\Delta- a{X_{s}} \geq 0$.
Since $K_{X/V} + \Delta$ is $\pi$-big, we have
\[
K_{X/V} + \Delta \sim_{\Q, \pi} A + E 
\]
for a $\pi$-ample $\Q$-Cartier divisor $A$ and an effective $\Q$-Cartier divisor $E$.
Take a rational number $\varepsilon \in \Q$ with $0< \varepsilon << 1$ such that the coefficients of
$\Delta- a X_{s} + \varepsilon E$ are less than one.
% \Q-Cartier にする？ a も Q ｄとるか
Note that 
\[
K_{X/V} + \Delta - a X_{s} + \varepsilon E + \varepsilon A
\sim_{\R, \pi} (1+ \varepsilon) (K_{X/V} + \Delta).
\]
Therefore, by using Proposition \ref{prop:cone theorem} for $B = \Delta - a X_{s} + \varepsilon E$, it finishes the proof.
\end{proof}

On the other hand, if the base scheme is local, then we can prove the cone theorem in a more general situation by reducing the problem to the special fiber.
In relative setting, since relative curve is not necessarily contained in the special fiber (e.g.\,the fiber of $X := \A^{1}_{\Z_{p}} \times \P^{1}_{\Z_{p}} \rightarrow \A^{1}_{\Z_{p}}$ over the closed point corresponding to $(pX+1)\subset \Z_{p}[X]$), we have to give an additional argument.
%Combining with the lifting argument (\cite[Theorem 1.3]{kawamata94}), we have the following proposition.

\begin{prop}
\label{prop:special fiber cone thm}
Let $V$ be an excellent Dedekind scheme.
Let $\pi \colon X \rightarrow U$ be a projective $V$-morphism from a normal $\Q$-factorial quasi-projective flat $V$-variety $X$ of relative dimension two to a quasi-projective  $V$-variety $U$.
%We suppose that the generic fiber $X_{\eta}$ is non-empty.
Let $B$ be an effective $\R$-divisor on $X$ such that every coefficient $c$ of $B$ satisfies $0 \leq c \leq 1$.
Let $A$ be a $\pi$-ample $\R$-Cartier $\R$-divisor on $X$.
We suppose one of the following.
\begin{enumerate}
    \item
The scheme $V$ is the spectrum of a discrete valuation ring.
    \item
The scheme $X$ is smooth over every generic point $\eta$ of $V$, and $B$ has no horizontal components.
\end{enumerate}
Then there exist finitely many $\pi$-relative curves $C_{1}, \ldots, C_{r}$ on $X$ such that
\[
\NE (X/U) = \NE (X/U)_{K_{X/V}+B+A \geq 0} + \sum_{i=1}^{i=r} \R_{\geq 0} [C_{i}].
\]
\end{prop}
% flat over dvr　は必要かも?
\begin{proof}
We may assume that $V$ is connected.
Moreover, replacing $\pi$ by its Stein factorization, we may assume that $\pi_{\ast} (\sO_{X}) = \sO_{U}$.
%If $U$ maps to the generic point $\eta \in V$, then the theorem follows from the case where base scheme is a characteristic $0$ field (see \cite{}).
%%% よく考えるとここも・・・・ 一般には char 0 ではない。 H-W って、 generic fiber に乗ってる場合どう扱った？→ delta = 0 case になるけど。。。  ・・・ relative dimension 2 とかで書いたほうがいい？？？？
Therefore we may assume $U$ is normal and flat over $V$.
First, we prove the assertion in the case (1).
%In this case, $V$ is the spectrum of principal ideal domain.
%So we may assume $\lfloor B \rfloor$ contains $X_{s}$ as sets, for any closed points $s$.
 If $U  \rightarrow V$ is not surjective, then $X \rightarrow U$ is normal surface over a field, so the assertion follows from the cone theorem for surfaces (cf.\,\cite[Theorem 2.14]{tanaka18}). Therefore we may assume $U \rightarrow V$ is surjective.
We denote the closed point of $V$ by $s$.
Let $S_{1}, \ldots, S_{n}$ be irreducible components of $X_{s}$.
Let $\nu_{i} \colon S_{i}^{N} \rightarrow S_{i}$ be the normalization.
By \cite[Proposition 4.5]{kollar13}, there exists an effective $\R$-divisor $D_{i}$ such that 
\[
K_{S^N_{i}} +  D_{i} \sim_{\R} (K_{X/ V} + B) |_{S_{i}^{N}}.
\]
Here, we put $D_{i}$ as $\Diff_{S_{i}^{N}} (B-S_{i}+ \alpha X_{s})$ for suitable $\alpha \in \R \geq 0$.
By \cite[Theorem 2.14]{tanaka18}, there exists finitely many $\pi \circ \nu_{i}$-relative curves $\Gamma_{i,j}$ such that
\[
\NE (S^N_{i}/U) = \NE (S^N_{i}/U)_{K_{S^{N}_{i}}+ D_{i} + A|_{S_{i}^{N}} \geq 0} + \sum_{j} \R_{\geq 0} [\Gamma_{i,j}].
\]
Now we will divide the case by the dimension of $U$.
First, consider the case where $U$ is of relative dimension $0$ over $V$.
%If $U  \rightarrow V$ is not surjective, then $X \rightarrow U$ is normal surface over a field, so the assertion follows from the cone theorem for surfaces (cf\. \cite[Theorem 2.14]{tanaka18}). So we may assume $U\ rightarrow V$ is surjective.
In this case, a closed curve in $X$ maps to a closed point in $U$, which maps to the closed point in $V$. Therefore, any $\pi$-relative curves are contained in $X_{s}$.
Therefore we have 
\begin{eqnarray*}
\NE(X/U) &=& \sum_{i=1}^{n} \NE(S^N_{i}/U) \\
&=& \NE(X/U)_{K_{X/V}+B+A \geq 0} + \sum_{i,j} \R_{\geq 0} [\Gamma_{i,j}].
\end{eqnarray*}
Next, we consider the case where $U$ is of relative dimension $1$ over $V$.
Let $\pi_{\eta} : X_{\eta} \rightarrow U_{\eta}$
be the restriction of $\eta$ to the generic fiber.
Then the generic fiber of $\pi_{\eta}$ is geometrically irreducible (cf.\,\cite[Lemma 2.2]{tanaka18a}).
%Since $U_{\eta}$ is a normal curve over $\eta$, there exist finitely many $\pi_{\eta}$-relative curves $C_{1}, \ldots C_{m}$ such that any $\pi_{\eta}$-relative curve $C$ is numerically equivalent to some $C_{i}$.
%Let $C$ be a $\pi$-relative curve which are contained in $X_{\eta}$.
Take an open subset $U_{0} \subset U_{\eta}$ where $\pi_{\eta}$ have geometrically irreducible fibers over $U_{0}$.
Take a closed point $P_{1} \in U_{0}$ which is also closed in $U$ if exists.
Let $P_{2}, \ldots P_{l} \in U_{\eta} \setminus U_{0}$ be all the closed points which are also closed in $U$.
Let $C_{s,t}$ be the irreducible components of $\pi_{\eta}^{-1} (P_{s})$.
Then any $\pi$-relative curve which is contained in $X_{\eta}$ is generated by $[C_{s,t}] \subset \NE(X/U)$.
Therefore, we have
\[
\NE (X/U) = \NE (X/U) _{K_{X/V}+ B+ A \geq 0} + \sum_{i,j} \R_{\geq 0} [\Gamma_{i,j}] + \sum_{s,t} \R_{\geq 0} [C_{s,t}]. 
\]
Finally, we consider the case where $U$ is of relative dimension $2$ over $V$.
In this case, $\pi$ is birational morphism.
Let $C_{1}, \ldots C_{l}$ be all the exceptional divisors of $\pi$ which are contained in the generic fiber $X_{\eta}$. Then $C_{i}$ are $\pi$-relative curves on $X$.
Then we have
\begin{eqnarray*}
\NE (X/U) &=& \NE (X_{s}/U) + \sum_{s} \R_{\geq0} [C_{s}] \\
&=& \NE (X/U) _{K_{X/V}+ B+ A \geq 0} + \sum_{i,j} \R_{\geq 0} [\Gamma_{i,j}] + \sum_{s} \R_{\geq 0} [C_{s}].
\end{eqnarray*}
It finishes the proof of (1).
The assertion in the case (2) follows from the argument in (1) and the lifting method in the proof of \cite[Theorem 1.3]{kawamata94}.
Indeed, since
\[
\NE(X/U) = \sum_{s\in |V|} \NE(X^{(s)}/U^{(s)})
\]
(where $X^{(s)}$ and $U^{(s)}$ are localization of $X$ and $U$ at a closed point $s$),
by the argument in (1), it suffices to show that any extremal ray in $\NE(X_{s}/U_{s})$ can be lifted to the generic fiber.
This follows from the deformation theory as in the proof of \cite[Theorem 1.3]{kawamata94}.
\end{proof}

\begin{prop}\label{prop:special termination}\textup{(cf.\,\cite[Theorem 4.2.1]{fujino07})}
Let $V$ be an excellent Dedekind scheme.
Let $(X,B)$ be a $\Q$-factorial three-dimensional dlt pair over $V$.
Consider a sequence of log flips starting from $(X,B)=(X_0,B_0)$:
\[
\xymatrix{
(X_0,B_0) \ar@{-->}[r] & (X_1,B_1) \ar@{-->}[r] & (X_2,B_2) \ar@{-->}[r] & \cdots,
}
\]
where $\phi_i \colon X_i \to Z_i$ is a flipping contraction associated to an extremal ray and $\phi^+ \colon X_i^+=X_{i+1} \to Z_i$ is the log flip.
Then, after finitely many flips, the flipping locus is disjoint from $\lfloor B \rfloor$.
\end{prop}

\begin{proof}
It follows from a similar argument to the argument in the proof of \cite[Theorem 4.2.1]{fujino07}.
\end{proof}

\begin{thm}\label{thm:birational relative mmp'}\textup{(Theorem \ref{thm:birational relative mmp}, cf.\,\cite[Theorem 1.1]{hacon-witaszek19})}
Let $V$ be an excellent Dedekind scheme.
Let $(X,\Delta)$ be a three-dimensional $\Q$-factorial dlt pair over $V$.
Assume that there exists a projective birational morphism $\pi \colon X \to Z$ to a normal $\Q$-factorial variety $Z$ with $\Exc(\pi) \subset \lfloor \Delta \rfloor$.
Then we can run a $(K_{X/V}+\Delta)$-MMP over $Z$ which terminates with a minimal model.
\end{thm}

\begin{proof}
It follows from the same argument as in the proof of \cite[Theorem 1.1]{hacon-witaszek19} using the cone theorem (Proposition \ref{prop:cone theorem}), the contraction theorem (Proposition \ref{prop:contraction anti-ample case}), the existence of flips (Corollary \ref{cor:existence of pl-flip threefold}), and the termination of flips (Proposition \ref{prop:special termination}). 
\end{proof}

\begin{lem}\label{lem:resolution}
Let $f \colon Y \to X$ be a projective birational morphism of three dimensional separated excellent integral schemes.
Let $W$ be a closed subscheme of $Y$.
If the singular locus of $X$ is contained in some affine open subset of $X$, then there exists a projective birational morphism $ \nu \colon Y' \to Y$ such that $Y'$ is regular and $\mathrm{Exc}(\nu) \cup \nu^{-1}(W)$ has simple normal crossing support.
\end{lem}

\begin{proof}
By \cite[Theorem 1.1]{cossart-piltant}, the scheme $X$ admits a projective resolution $X_0$.
By \cite[Theorem 4.4]{cossart-piltant} and \cite{cjs}, an analog of \cite[Conjecture 5.4]{hacon-witaszek20} holds for regular three dimensional separated excellent integral schemes and its closed subschemes.
By an analogous argument of the proof of \cite[Proposition 5.5]{hacon-witaszek20} for $X_0 \dashrightarrow Y$, we obtain the assertion.
We note that we can take an elimination of $X_0 \dashrightarrow Y$ which is projective over $X_0$ and $Y$.
\end{proof}

\begin{rmk}\label{rem:resolution}
Let $U$ be a regular open subset of $Y$ such that $f|_{U}$ is an isomorphism and $\mathrm{Exc}(\nu) \cup \nu^{-1}(W)$ has simple normal crossing support on $U$.
We can take a morphism $\nu$ in Lemma \ref{lem:resolution} as a morphism whose isomorphic locus contains $U$.
Indeed, there exists an elimination of $X_0 \dashrightarrow Y$ whose isomorphic locus contains $U$. 
\end{rmk}

\begin{cor}\label{cor:dlt modification}\textup{(\cite[Corollary 1.4]{hacon-witaszek19})}
Let $V$ be an excellent Dedekind separated scheme.
Let $(X,\Delta)$ be a $\Q$-factorial three-dimensional log pair over $V$ such that any coefficient of $\Delta$ is at most one.
Assume that $X$ is projective birational over an affine $V$-variety.
Then there exits a projective birational morphism $\pi \colon Y \to X$ such that the pair $(Y,\Delta_Y:=\pi^{-1}_*\Delta+\Exc(\pi))$ satisfies the following conditions.
\begin{enumerate}
    \item $(Y,\Delta_Y)$ is a $\Q$-factorial dlt pair over $V$, and
    \item $K_{Y/V}+\Delta_Y$ is nef over $X$.
%    \item $\pi$ is an isomorphism over the dlt locus of $(X,\Delta)$.
\end{enumerate}
\end{cor}

\begin{proof}
By Lemma \ref{lem:resolution},
we have a log resolution $f \colon W \to X$ of $(X,\Delta)$.
%which is an isomorphism over the simple normal crossing locus of $(X,\Delta)$.
We write $\Delta_W:=\pi^{-1}_*\Delta+\Exc(\pi)$, then $(W,\Delta_W)$ is dlt.
By Theorem \ref{thm:birational relative mmp}, we can run a $(K_{W/V}+\Delta_W)$-MMP over $X$, and we get a minimal model $\pi \colon Y \to X$.
Then $(Y,\Delta_Y:=\pi^{-1}_*\Delta +\Exc(\pi))$ is dlt and $K_{Y/V}+\Delta_Y$ is nef over $X$.
%Let $E$ be an exceptional prime divisor on $Y$.
%The dlt locus of $(X,\Delta)$ is denoted by $U$.
%By the negativity lemma (Proposition \ref{prop:negativity lemma}), we have
%\[
%K_{Y/V}+\Delta_Y-\pi^*(K_X+\Delta) \leq 0,
%\]
%so the log discrepancy $a_E(X,\Delta) $ is non-positive.
%If the center $c_X(E)$ intersects $U$, then $(X,\Delta)$ is a simple normal crossing pair around the generic point $x$ of $c_X(E)$ by the definition of dlt.
%Since $f$ is an isomorphism at $x$, so is $\pi$.
%Thus, $\pi$ is small over $U$, and in particular $\pi$ is an isomorphism because $X$ is $\Q$-factorial.
\end{proof}

\begin{cor}\label{cor:inversion of adjunction}\textup{(cf.\, \cite[Corollary 1.5]{hacon-witaszek19})}
Let $V$ be an excellent Dedekind separated scheme.
Let $(X,S+B)$ be a $\Q$-factorial three-dimensional log pair over $V$ such that $S$ is a locally irreducible Weil divisor.
Assume that $X$ is projective birational over an affine $V$-variety.
Then $(X,S+B)$ is plt on a neighborhood of $S$ if and only if $(S^N,B_S)$ is klt, where $S^N$ is the normalization of $S$ and $B_S:=\mathrm{Diff}_{S^N}(B)$ is the different.
\end{cor}

\begin{proof}
It follows from the same argument as in the proof of \cite[Corollary 1.5]{hacon-witaszek19} replacing \cite[Corollary 1.4]{hacon-witaszek19} into Corollary \ref{cor:dlt modification}.
\end{proof}

\begin{lem}\label{lem:trivial case}
Let $V$ be an excellent Dedekind scheme.
Let $(X,\Delta)$ be a three-dimensional dlt pair over $V$ with $\lfloor \Delta \rfloor=S+S'$, where $S$ and $S'$ are locally irreducible $\Q$-Cartier divisors.
Let $f \colon X \to Z$ be a projective morphism over $V$ such that $f(S)$ is two-dimensional.
Let $C$ be an irreducible component of $S \cap S'$.
If $f(C)$ is a point, then $(S' \cdot C)$ is negative.
\end{lem}

\begin{proof}
We set $D:=\mathrm{Diff}_{S^N}(\Delta-S)$, then $(S^N,D)$ is plt by Proposition \ref{prop:dlt adjunction to plt}.
As $S^N$ is two-dimensional, $\lfloor D \rfloor$ is normal, and in particular, $\lfloor D \rfloor$ is locally irreducible.
Since $(X,\Delta)$ is a simple normal crossing pair on the generic point of $S \cap S'$, we have $\lfloor D \rfloor=S'|_{S^N}$.
By Lemma \ref{lem:topologically normal}, $S^N \to S$ is a universal homeomorphism, thus $S \cap S'$ is also locally irreducible.
In particular, $C|_S$ is $\Q$-Cartier and $(C|_{S^N} \cdot \lfloor D\rfloor )=(C|_{S^N}^2)$.
Thus, we have
\[
(C \cdot S')=(C|_S \cdot S'|_S) = (C|_{S^N} \cdot \lfloor D \rfloor)=(C|_{S^N}^2) < 0,
\]
because $C|_{S^N}$ is a contracted curve of $S^N$ via generically finite morphism $f|_{S^N} \colon S^N \to f(S)$.
\end{proof}

\begin{prop}\label{prop:existence of flip trivial case}\textup{(cf.\,\cite[Proposition 4.1]{hacon-witaszek19})}
Let $V$ be the spectrum of an excellent discrete valuation ring.
Let $X$ be a flat $V$-variety of relative dimension two.
Let $(X,\Delta)$ be a dlt $\Q$-factorial log pair over $V$ .
Let $f \colon X \to Z$ be a $(K_{X/V}+\Delta)$-flipping contraction with $\rho(X/Z)=1$.
Suppose that $X_s$ is contained in $\lfloor \Delta \rfloor$ as sets and every irreducible component of $X_s$ is numerically trivial over $Z$.
Then the flip of $f$ exists.

\end{prop}

\begin{proof}
We may assume that $\lfloor \Delta \rfloor=X_{s,\mathrm{red}}$.
We take a point $z \in f(\Exc(f))$.
By shrinking $Z$, we may assume that $Z$ is affine.
First, we prove that $f^{-1}(z)$ intersects with only one irreducible component of $X_s$.
Otherwise, there exist two irreducible components $S$ and $S'$ intersecting $f^{-1}(z)$.
By the connectedness of $f^{-1}(z)$, there exists a flipping curve $C$ intersecting with $S$ and $S'$.
By the assumption, $S$ and $S'$ are numerically trivial over $Z$.
Since $(S \cdot C)=0$ and $(S' \cdot C)=0$, $C$ is contained in $S$ and $S'$.
It contradicts Lemma \ref{lem:trivial case}.

Thus, we may assume that $X_s$ is irreducible by shrinking $Z$ around $z$, and in particular, $(X,\Delta)$ is plt.
We take a reduced $\Q$-Cartier divisor $H$ on $X$ as in \cite[Lemma 5.3]{hacon-witaszek20}, then $H \equiv_Z 0$ and it satisfies the conditions in the proof of \cite[Theorem 1.2]{hacon-witaszek20}.
We note that as $X_\eta$ is a surface and the quotient field of $V$ is infinite,  a general hyperplane preserves dlt singularities on the generic fiber.
We take a dlt modification $Y \to X$ of $(X,\Delta+H)$ by Corollary \ref{cor:dlt modification}.
We note that $f \colon X \to Z$ is an isomorphism over the generic point of $V$, we may assume that $Y \to X$ is an isomorphism over the generic point by Remark \ref{rem:resolution}.
We run a $(K_{Y/V}+\Delta_Y+H_Y)$-MMP by the same argument as in \cite[Theorem 1.2]{hacon-witaszek20}.
Replacing $(Y,\Delta_Y+H_Y)$ into a minimal model, we may assume that $K_{Y/V}+\Delta_Y+H_Y$ is nef over $Z$.
By Corollary \ref{cor:mmp with scaling} and the same argument as in the proof of \cite[Theorem 1.2]{hacon-witaszek20}, we can run a $(K_{Y/V}+\Delta_Y)$-MMP over $Z$ with scaling of $H_Y$.
Replacing $(Y,\Delta_Y)$ into a minimal model, we may assume that $K_{Y/V}+\Delta_Y$ is nef over $Z$.
We denote the map $Y \to Z$ by $h$.
Since $(X,\Delta)$ is plt, $h$ is small by the negativity lemma (Proposition \ref{prop:negativity lemma}).
Since the relative Picard rank of $X$ over $Z$ is one, $Y$ is the flip of $f$.
\end{proof}

\section{Proof of Theorem \ref{thm:semi-stable mmp} and its applications}\label{section:dedekind scheme}
Our goal of this section is to prove Theorem \ref{thm:semi-stable mmp}, which is a generalization of the result of Kawamata (\cite{kawamata94}, \cite{kawamata99}).
In this section, we deal with schemes satisfying the following conditions (Assumption \ref{assump:semi-stable}), which are preserved under MMP-steps (cf.\,Proposition \ref{prop:assumption under divisorial contraction}, \ref{prop:assumption under flip}).
Kawamata proved this fact by the construction of flips, but it does not follow from our construction.
Therefore, to prove the preservation, we precisely observe extremal ray contractions (cf.\,Proposition \ref{prop:contraction theorem semi-stable}).

\begin{assump}\label{assump:semi-stable}
Let $V$ be an excellent Dedekind scheme.
$X$ is a $V$-variety satisfying the following conditions.
\begin{enumerate}
    \item $X$ is flat over $V$ of relative dimension two.
    \item Every generic fiber $X_\eta$ is smooth.
    \item The fibers $X_s$ for the closed points $s \in V$ are geometrically reduced and satisfy the condition $(S_2)$.
    \item Each irreducible component $S$ of every fiber $X_s$ is geometrically irreducible, geometrically normal and a $\Q$-Cartier divisor on $X$.
    \item $(X,X_s)$ is dlt for all closed points $s \in V$.
    \item For each closed point $s \in V$ and dominant morphism $\iota \colon V'=\mathrm{Spec} A \to V$ with reduced fiber such that $A$ is a complete discrete valuation ring with algebraically closed residue field $k$ and $s$ is contained in the image of $\iota$, the base change $X' := X \times_{V} V'$ satisfies the condition  $(5)$.
\end{enumerate}
\end{assump}

\begin{rmk}
\begin{itemize}
\item
A strictly semi-stable scheme over an excellent Dedekind scheme of relative dimension $2$ satisfies Assumption \ref{assump:semi-stable}.
    \item The existence of an extension as in Assumption \ref{assump:semi-stable} $(6)$ follows from \cite[Theorem 29.1]{matsumura}.
    \item Assumption \ref{assump:semi-stable} is preserved by taking a base change as in Assumption     \ref{assump:semi-stable} $(6)$.
\end{itemize}
\end{rmk}
 
\begin{rmk}
In \cite{kawamata94}, it is additionally assumed that $\sO_X(mK_{X/V})$ is maximal Cohen-Macaulay in order to prove the existence of flips.
Kawamata proved that the condition is preserved under MMP-steps if each residue characteristic is larger than $3$.
However, in this paper, we do not need this assumption.
We note that if each residue characteristic is larger than $5$, then such a condition is induced by Assumption \ref{assump:semi-stable}.
Indeed, we take a closed point $x \in X$ contained in an irreducible component $S$ of some closed fiber $X_s$.
Then by Assumption \ref{assump:semi-stable} $(4)$ and $(5)$, $(X,S)$ is plt.
Thus, by the adjunction, $S$ is a klt surface.
Since the characteristic of $S$ is larger than $5$, $S$ is strongly $F$-regular.
By Proposition \ref{cor:inversion of adjunction}, $X$ is $T$-regular at $x$.
In conclusion, $X$ is $T$-regular.
By Proposition \ref{prop:vanishing for T-regular}, $\sO_X(mK_{X/V})$ is maximal Cohen-Macaulay.
\end{rmk}

\begin{lem}\label{lem:reducedness}
Let $V$ be an excellent Dedekind scheme and $X$ is a $V$-variety satisfying Assumption \ref{assump:semi-stable}.
Let $s \in V$ be a closed point.
Let $S_1, \ldots S_r$ be the irreducible components of $X_s$.
We set $X_i:=S_1 \cup \cdots \cup S_i$ with reduced structure and the scheme-theoretic intersection $C_i := X_{i-1} \cap S_i$ for $1 \leq i \leq r$.
Then $X_i$ is reduced and satisfies the condition $(S_2)$ and $C_i$ is reduced and pure one-dimensional for every $i$.
\end{lem}

\begin{proof}
Since $X_s$ is reduced and $X_s=X_r$ as sets, we have $X_s=X_r$.
In particular, $X_r$ satisfies the condition $(S_2)$.
Since $X_{i-1}$ and $S_i$ are $\Q$-Cartier divisors on $X$, the scheme-theoretic intersection $C_i=X_{i-1} \cap S_i$ is pure one-dimensional.
In particular, each generic point of $C_i$ is codimension two point in $X$.
Since $(X,X_s)$ is a simple normal crossing pair at each generic point of $C_i$, the scheme $C_i$ satisfies the condition $(R_0)$.
Thus, in order to prove $C_i$ is reduced, it is enough to show that $C_i$ satisfies the condition $(S_1)$.

We take a closed point $P$ of $C_i$.
Then $P$ is contained in at least two components $S_i$ and $S_j$ for some $j <i$.
If $P$ is contained in three components, then $(X,X_s)$ is a simple normal crossing pair at $P$, and in particular, $C_i$ is reduced at $P$.
Thus, we may assume that $P$ is contained in only two components, so $X_s=S_i \cap S_j$ around $P$ and we obtain the exact sequence
\[
0 \to \sO_{X_s} \to \sO_{S_i} \oplus \sO_{S_j} \to \sO_{C_i} \to 0
\]
around $P$.
Since $X_s$, $S_i$ and $S_j$ satisfy the condition $(S_2)$, $C_i$ satisfies the condition $(S_1)$, so $C_i$ is reduced for all $i$.
The exact sequence
\[
0 \to \sO_{X_i}\to \sO_{X_{i-1}} \oplus \sO_{S_i} \to \sO_{C_i} \to 0
\]
implies that if $X_{i-1}$ satisfies the condition $(S_2)$, then so is $X_i$ by the reducedness of $C_i$.  
\end{proof}

\begin{prop}\label{prop:connectedness}
Let $\pi \colon S \to Z$ be a projective morphism from a surface to a variety over an algebraically closed field $k$.
Let $(S,D)$ be a dlt pair and $L$ be a $\pi$-nef Cartier divisor such that $L-(K_S+D)$ is $\pi$-ample and $C$ be a reduced Weil divisor with $C \leq D$.
Then the following hold.
\begin{itemize}
    \item[$(i)$] $\pi_*\sO_S(mL) \to \pi_*\sO_C(mL)$ is surjective for all $i$ and divisible enough $m$.
    \item[$(ii)$] $R^i\pi_*\sO_S(mL)=R^i\pi_*\sO_C(mL)=0$  for every $i > 0$ and divisible enough $m$.
    \item[$(iii)$] $L$ is semiample over $Z$.
\end{itemize}

\end{prop}

\begin{proof}
We note that the semiampleness follows from the abundance, since $L-(K_S+D)$ is ample over $\pi$ and $k$ is infinite.
By \cite[Theorem 1.1]{tanaka20-imperfect}, $L$ is semiample over $Z$, thus we may assume that $L$ is a pullback of an ample Cartier divisor on $Z'$, where $\pi' \colon S \to Z'$ is the morphism defined by $L$ over $Z$.
In particular, we may assume that $L$ is trivial by replacing $\pi$ into $\pi'$.
First, we consider the case where the dimension of $\pi(S)$ is at least one.
By a perturbation of coefficients of $D$ and \cite[Lemma 2.1]{kawamata94}, we have 
\[
R^i\pi_*\sO_S=R^i\pi_*\sO_S(-C)=0
\]
for all $i>0$, thus we obtain the assertion.
Next, we assume that $\pi(S)$ is a point.
By \cite[Lemma 2.2]{kawamata94}, we have $H^i(\sO_S)=0$ for all $i>0$.
By \cite[Theorem 5.2]{tanaka18}, the scheme $C$ is connected, so we have $H^0(\sO_C) \simeq k$.
Since we have $H^0(\sO_S(-C))=0$ and $H^0(\sO_S) \simeq k$, the map $H^0(\sO_S) \to H^0(\sO_C)$ is surjective.
Thus we have $H^1(\sO_S(-C))=0$.
Since $-(K_S+C)$ is big, we have $H^2(\sO_S(-C))=0$, so we have $H^1(\sO_C)=0$,
\end{proof}

The following theorem is discussed in \cite[Theorem 2.3]{kawamata94}.
However, we need more detailed observation of contractions, thus we use a bit different method form the method of \cite[Theorem 2.3]{kawamata94}.

\begin{prop}\label{prop:contraction theorem semi-stable}\textup{(cf.\,\cite[Theorem 2.3]{kawamata94})}
Let $V$ be an excellent Dedekind scheme and $X$ is a $V$-variety satisfying Assumption \ref{assump:semi-stable}.
Let $\pi \colon X \to U$ be a projective morphism to a $V$-variety.
Let $s \in V$ be a closed point.
Let $S_1, \ldots S_r$ be the irreducible components of $X_s$.
Let $L$ be a $\rho$-nef Cartier divisor with $L-K_{X/V}$ is $\pi$-ample.
Then $L$ is semiample.
Furthermore, the map $f \colon X \to Z$ defined by $L$ satisfies the following conditions.
\begin{enumerate}
    \item  $R^jf_*\sO_X=0$ for all $j>0$.
    \item $\pi_*\sO_{S_i}=\sO_{f(S_i)}$, where the images are equipped with the reduced structure.
    \item $Z$ is Cohen-Macaulay.
\end{enumerate}
\end{prop}

\begin{proof}
Taking a base change via $\iota \colon V' \to V$ as in Assumption \ref{assump:semi-stable} $(6)$,
we may assume that $V$ is the spectrum of a complete discrete valuation ring with algebraically closed residue field.
We set $X_i:=S_1 \cup \cdots \cup S_i$ and $C_i := X_{i-1} \cap S_i$ for $1 \leq i \leq r$.
We may assume that the conditions $(i), (ii)$ in Proposition \ref{prop:connectedness} for $S_i$ are satisfied for $m=1$ and $R^j\pi_*\sO_{X_\eta}(L)=0$ by replacing $L$ into some power of $L$.
We note that $X_i$ is a pushout of $X_{i-1}$ and $S_i$ with $C_i$, and $C_i$ is reduced and $X_i$ satisfies $(S_2)$ for all $i$ by Lemma \ref{lem:reducedness}.
Thus, we have the surjection $\pi_*\sO_{X_i}(L) \to \pi_*\sO_{X_{i-1}}(L)$ and the isomorphism $R^j\pi_*\sO_{X_i}(L) \simeq R^j\pi_*\sO_{X_{i-1}}(L)$ by Proposition \ref{prop:connectedness}.
By the induction on $i$ and changing the order of $S_1 , \ldots, S_r$, we have  $R^j\pi_*\sO_{X_s}(L)=0$ for all $j>0$ and the surjection $\pi_*\sO_{X_i}(L) \to \pi_*\sO_{S_i}(L)$ for all $i$.
By the sujectivity of $\pi_*\sO_{X_i}(L) \to \pi_*\sO_{S_i}(L)$ and $\pi_*\sO_{X_i}(L) \to \pi_*\sO_{X_{i-1}}(L)$, if $L|_{S_i}$ and $L|_{X_{i-1}}$ is globally generated over $U$, then so is $L|_{X_i}$.
By the induction on $i$, we may assume that $L|_{X_s}$ is globally generated.

By the exact sequence
\[
0 \to \sO_X(L) \to \sO_X(L) \to \sO_{X_s}(L) \to 0,
\]
where the first map is the multiplication by a uniformizer $\varpi$, we have the surjection
\[
\xymatrix{
R^j\pi_*\sO_X(L) \ar[r]^{\cdot \varpi} & R^j\pi_*\sO_X(L)
}
\]
for all $j>0$.
Thus, this vanishes around the closed fiber and the generic fiber, we have $R^j\pi_*\sO_X(L)=0$ for all $j>0$.
Since $L|_{X_s}$ and $L|_{X_\eta}$ is semiample and we have the surjection $\pi_*\sO_X(L) \to \pi_*\sO_{X_s}(L)$, $L$ is semiample over $U$.

Next, we prove the assertions $(1)$ and $(2)$, so we may assume that $L$ is trivial.
By the above argument, we have $R^jf_*\sO_X=0$ for all $j>0$ and  $f_*\sO_X \to f_*\sO_{S_i}$ is surjective for all $i$.
By the commutative diagram
\[
\xymatrix{
f_*\sO_X \ar[r]^{\simeq} \ar@{->>}[d] & \sO_Z \ar@{->>}[d] \\
f_*\sO_{S_i} \ar@{^(->}[r] & \sO_{f(S_i)},
}
\]
we have $f_*\sO_{S_i} =\sO_{f(S_i)}$.
By the same argument as above, we have $f_*\sO_{X_i}=\sO_{f(X_i)}$ and $f_*\sO_{C_i}=\sO_{f(C_i)}$ for all $i$.
By the vanishing $R^1f_*\sO_{X_i}=0$ for all $i$, we have the exact sequence
\[
0 \to \sO_{f(X_i)} \to \sO_{f(S_i)} \oplus \sO_{f(X_{i-1})} \to \sO_{f(C_i)} \to 0.
\]
By the induction and the condition $(S_1)$ on $f(C_i)$, the scheme $f(X_i)$ satisfies the condition $(S_2)$.
In particular, $Z_s$ satisfies the condition $(S_2)$, thus $Z$ is Cohen-Macaulay.
\end{proof}

\begin{prop}\label{prop:assumption under divisorial contraction}
Let $V$ be an excellent Dedekind scheme and $X$ is a $\Q$-factorial $V$-variety satisfying Assumption \ref{assump:semi-stable}.
Let $f \colon X \to Z$ be a $K_X$-negative extremal ray contraction which is a divisorial contraction.
Then $Z$  also satisfies Assumption \ref{assump:semi-stable}.
\end{prop}

\begin{proof}
By Proposition \ref{prop:contraction theorem semi-stable}, $Z$ is Cohen-Macaulay, thus $Z$ satisfies Assumption \ref{assump:semi-stable} $(3)$.
The other conditions follow from the standard argument.
\end{proof}

\begin{prop}\label{prop:assumption under pl-flip}
Let $V$ be the spectrum of an excellent discrete valuation ring and $X$ is a $\Q$-factorial $V$-variety satisfying Assumption \ref{assump:semi-stable}.
Let 
\[
\xymatrix{
X \ar[r]^-\phi & Z & \ar[l]_-{\phi^+} Y
}
\]
be a $K_X$-flip with $\rho(X/Z)=1$.
Assume that there exists irreducible components $S$ and $A$ of the closed fiber $X_s$ such that $-S$ and $A$ are ample.
Then the strict transforms $S'$ and $A'$ of $S$ and $A$ on $Y$, respectively, are geometrically normal.
\end{prop}

\begin{proof}
Taking a base change via $V' \to V$ in Assumption \ref{assump:semi-stable}, we may assume that $V$ is the spectrum of a complete discrete valuation ring with algebraically closed residue field.
Since $A$ and $S$ are geometrically irreducible, the irreducibility is preserved.
We take a point $z \in \phi(\Exc(\phi))$, by shrinking $Z$ around $z$, we may assume that $\phi(\Exc(\phi))=\{z\}$ and $Z$ is affine.
We denote the image of $S$ and $A$ via $\phi$ by $T$ and $B$, respectively.
By Proposition \ref{prop:contraction theorem semi-stable}, $\phi_*\sO_S=\sO_T$ and $\phi_*\sO_A=\sO_B$, so $T$ and $B$ are normal.
By the proof of Lemma \ref{lem:bcm regular and globally $T$-regular surface}, we may assume that the intersection $C:=A \cap S$ is irreducible and not contracted by shrinking $Z$ around $z$.
By the perturbation of coefficients, there exists a dlt pair $(X,\Delta)$ such that $\lfloor \Delta \rfloor =A+S$ and $-(K_{X/V}+\Delta)$ is ample.
The strict transform of $\Delta$ on $Y$ is denoted by $\Delta_Y$.
Since $X$ and $Y$ have isolated singularities, the different coincides with $(\Delta-S)|_S$ and $(\Delta_Y-S')|_{S'^N}$.
In particular, $(\phi|_S)_*(\Delta-S)|_S=(\phi^+|_{S'^N})_*(\Delta_Y-S')|_{S'^N}$, it is denoted by $D_T$.
Furthermore, the divisor
\[
K_{S'^N}+(\Delta_Y-S')|_{S'^N}=(K_{Y/V}+\Delta_Y)|_{S'^N}
\]
is ample.
%By Lemma \ref{lem:pl-flip good condition}, $(S,\Delta|_S)$ is purely globally $T$-regular.
We denote $A|_S$ by $C$ and the image of $C$ by $C'$ via $\phi|_S$.
Since $C \to C'$ is an isomorphism by the proof of Lemma \ref{lem:bcm regular and globally $T$-regular surface}, $(C',\mathrm{Diff}(D_T))$ is globally $T$-regular.
Therefore, by Corollary \ref{cor:inversion of adjunction}, $(T,D_T)$ is purely globally $T$-regular.
%Then by the proof of Lemma \ref{lem:bcm regular and globally $T$-regular surface}, $(T,\Delta_T)$ is globally $T$-regular.
Since $K_{S'^N}+(\Delta_Y-S')|_{S'^N}$ is ample,
there exists an effective divisor $\Gamma$ of $S'^N$ such that
\[
(\phi^+|_{S'^N})^*(K_T+D_T)=K_{S'^N}+(\Delta_Y-S')|_{S'^N}+\Gamma.
\]
Since Proposition \ref{prop:crepant globally $T$-regular}, the pair $(S'^N,K_{S'^N}+(\Delta_Y-S')|_{S'^N}+\Gamma)$ is globally $T$-regular, thus $(S'^N,K_{S'^N}+(\Delta_Y-S')|_{S'^N})$ is also globally $T$-regular.
Therefore, $S'$ is normal.

Next, we consider the normality of $A'$.
By the above argument, $S|_A$ is not exceptional, irreducible and anti-ample over $B$.
Thus, $\phi|_A$ is finite, and in particular, $\phi|_A$ is an isomorphism because of $\phi_*\sO_A=\sO_B$.
By the same argument as above, $(A'^N,(\Delta_Y-S')|_{A'^N})$ is globally $T$-regular, thus $A'$ is normal.
\end{proof}

\begin{prop}\label{prop:assumption under flip}
Let $V$ be an excellent Dedekind scheme and $X$ is a $\Q$-factorial $V$-variety satisfying Assumption \ref{assump:semi-stable}.
Let 
\[
\xymatrix{
X \ar[r]^-\phi & Z & \ar[l]_-{\phi^+} Y
}
\]
be a $K_X$-flip with $\rho(X/Z)=1$.
Then $Y$ is a scheme satisfying Assumption \ref{assump:semi-stable}.
\end{prop}

\begin{proof}
It is obvious that $Y$ satisfies Assumption \ref{assump:semi-stable} except for the condition $(S_2)$ for closed fibers in $(3)$ and the geometric normality of irreducible components of closed fibers in $(4)$.
First, we prove that the $(S_2)$ condition of closed fibers, it is enough to show that $Y$ is Cohen-Macaulay.
We take a point $z \in \phi(\Exc(\phi))$, by shrinking $Z$ around $z$, we may assume that $\phi(\Exc(\phi))=\{z\}$ and $Z$ is affine.
In particular, we may assume that $V$ is a discrete valuation ring by localizing at the image $s$ of $z$.
By Proposition \ref{prop:contraction theorem semi-stable}, we have $R^1\phi_*\sO_X=0$  and $Z$ is Cohen-Macaulay.
Thus $Z$ has rational singularities by Proposition \ref{prop:rational singularities 1}.
Since $X$ is terminal, so is $Y$, and in particular, $Y$ has isolated singularities.
Since $\phi^+$ is small, $Y$ has also rational singularities by Proposition \ref{prop:rational singularities 2}.
Thus, $Y$ is Cohen-Macaulay.

Next, we prove the geometric normality of irreducible components of $X_s$.
First, we consider the case where $f^{-1}(z)$ is contained in only one irreducible component of $X_s$.
%Let $C$ be a flipping curve, that is, $C$ is a curve in $X$ such that $\phi(C)$ is a point.
%First, we consider the case where there exists an irreducible component $S$ of $X_s$ containing $C$ with $(C \cdot S)=0$.
%Since the relative Picard rank one, it is enough to show that $C$ does not intersect any other component.
%Otherwise, $C$ is contained in an irreducible component $S'$ of $X_s$.
%Since $(X,S+S')$ is dlt, $(S,D)$ is plt pair, where $D=S'|_S$.
%Then $C$ is a component of $D$ and contracted by a birational morphism $\phi_S \colon S \to T$, where $T=\phi(S)$.
%By \cite[Theorem 5.2]{tanaka18} and the locally irreducibility of $D$, we may assume $D=C$ by shrinking $Z$ around $z$.
%In particular, we have 
%\[
%S \cdot C = S|_{S'} \cdot C|_{S'} = (C|_{S'}^2) < 0,
%\]
%it contradicts to $(S \cdot C)=0$.
Then, we may assume that $X_s$ is irreducible.
Thus, $X_s$ is geometrically irreducible, and so is $Y_s$.
After taking base change $V' \to V$ as in Assumption \ref{assump:semi-stable} $(6)$, the pair $(Y,Y_s)$ is dlt and $Y_s$ is irreducible. 
Thus, $Y_s$ is normal in codimension one.
Since $Y$ is Cohen-Macaulay, so is $Y_s$, thus $Y_s$ is normal.
In conclusion, $Y_s$ is geometrically normal.

Next, we consider the other case.
The extremal ray contracted by $\phi$ is denoted by $\Sigma$.
We take an irreducible component $R$ of $Y_s$ such that $\phi^+(R)$ contains $z$.
We write $T:=\phi^+_*R$ and $S:=\phi^{-1}_*T$.
Suppose $(S \cdot \Sigma)=0$, then $f^{-1}(z)$ is contained in $S$.
By Lemma \ref{lem:trivial case}, every flipping curve is not contained in any other components.
By the assumption, $f^{-1}(z)$ intersects another component $S'$ of $X_s$.
If $(S' \cdot \Sigma) \leq 0$, then $S'$ contains flipping curves, so we have a contradiction.
If $(S' \cdot \Sigma)$ is positive, then replacing $S'$, we have $(S' \cdot \Sigma)$ is negative because of $X_s \sim 0$.
In conclusion, we obtain $(S \cdot \Sigma) \neq 0$.
First, we consider the case where $(S \cdot C)$ is positive.
Then there exists an irreducible component $S'$ of $X_s$ such that $(S' \cdot C)$ is negative.
In particular, the divisors $S$ and $-S'$ are ample.
By Proposition \ref{prop:assumption under pl-flip}, the scheme $R$ is geometrically normal.
On the other hand, if $(S \cdot C)$ is negative, then there exists an irreducible component which is ample.
Thus, by Proposition \ref{prop:assumption under pl-flip}, the scheme $R$ is geometrically normal.
\end{proof}

\begin{thm}\label{thm:semi-stable mmp'}\textup{(Theorem \ref{thm:semi-stable mmp}, cf.\,\cite{kawamata94})}
Let $V$ be an excellent Dedekind scheme.
Let $X$ be a $V$-variety satisfying Assumption \ref{assump:semi-stable}.
Then we can run a $K_{X/V}$-MMP over $Z$ preserving Assumption \ref{assump:semi-stable} which terminates with a minimal model or a Mori fiber space.
\end{thm}

\begin{proof}
We note that $(K_{X/V}+X_s)$-MMP is also $K_{X/V}$-MMP because $X_s$ is linearly trivial.
By the cone theorem (Proposition \ref{prop:special fiber cone thm}) and the contraction theorem (Proposition \ref{prop:contraction theorem semi-stable}), we can contract any $K_{X/V}$-negative extremal ray.
Let $f \colon X \to Z$ be a $K_X$-negative extremal ray contraction.
If $f$ is divisorial contraction, $Z$ also satisfies Assumption \ref{assump:semi-stable} by Proposition \ref{prop:assumption under divisorial contraction}.
If $f$ is flipping contraction, the extremal ray contracted by $f$ is denoted by $\Sigma$.
In order to prove the existence of the flip, we may assume that $V$ has the unique closed point $s$.
If $(S \cdot \Sigma)=0$ for every irreducible component $S$ of $X_s$, then the flip of $f$ exists by Proposition \ref{prop:existence of flip trivial case}.
Otherwise, as $X_s$ is linearly trivial, there exists irreducible components $S$ and $A$ of $X_s$ such that $(S \cdot \Sigma)< 0$ and $(A \cdot \Sigma)>0$.
By Corollary \ref{cor:existence of pl-flip threefold}, the flip of $f$ exists.

Then the flip $X \dashrightarrow Y$ of $f$ exists and $Y$ satisfies Assumption \ref{assump:semi-stable} by Proposition \ref{prop:assumption under flip}.
Since $X$ has terminal singularities, a sequence of flip terminates by the argument in \cite[Theorem 6.17]{kollar-mori}.
Thus, $K_{X/V}$-MMP terminates with a minimal model or a Mori fiber space.
\end{proof}

In the following, we review applications of Theorem \ref{thm:semi-stable mmp} which are discussed in \cite{chiarellotto-lazda}.
\begin{defn}[{cf.\, \cite[Definition 5.1]{chiarellotto-lazda}}]
\label{defn: minimalstrictsemist}
Let $\sO_{K}$ be an excellent Henselian discrete valuation ring with perfect residue field $k$.
Let $K$ be the fraction field of $\sO_{K}$.
Let $X$ be a K3 surface over $K$ or an abelian surface over $K$.
Here, we note that an abelian surface does not necessarily admit a section.
Then a \emph{minimal strictly semi-stable model} of $X$ is a proper algebraic space $\mathcal{X}$ over $\sO_{K}$ satisfying the following.
\begin{enumerate}
    \item The generic fiber $\mathcal{X}_{K}$ is isomorphic to $X$.
    \item The special fiber $\mathcal{X}_{s}$ is a scheme whose irreducible components are smooth over $k$.  
    \item There exists an \'{e}tale surjection $U \rightarrow \mathcal{X}$ such that $U$ is a strictly semi-stable scheme in the sense of Definition \ref{defn:semi-stable}.
    \item 
    The relative dualizing sheaf $\omega_{\mathcal{X}/\sO_{K}}$ is trivial. Here, see \cite[Section 5]{chiarellotto-lazda} for the definition of the relative dualizing sheaf. 
\end{enumerate}
\end{defn}

\begin{thm}
\label{thm: potentiallyminimalsemistable}
Let $\sO_{K}$, $K$, $k$ and $X$ be as in Definition \ref{defn: minimalstrictsemist}.
Suppose one of the following.
\begin{enumerate}
    \item 
The scheme $X$ is an abelian surface over $K$.
    \item
The scheme $X$ is a K3 surface over $K$ satisfying that $X$ admits a projective strictly semi-stable scheme model over $\sO_{K}$.
\end{enumerate}
Then there exists a finite separable extension $K' / K$ such that there exists a minimal strictly semi-stable model over $\sO_{K'}$ of $X_{K'}$. Moreover, in the case (1), we can take a finite separable extension $K'/K$ such that there exists a minimal strictly semi-stable scheme model over $\sO_{K'}$.
\end{thm}
\begin{proof}
By Theorem \ref{thm:semi-stable mmp}, the first half of this theorem follows from the proof of \cite[Theorem 10.3]{chiarellotto-lazda} and the proof of \cite[Proposition 2.1]{liedtke-matsumoto}.
Now we will consider the case (1).
By \cite{Kunnemann}, there exists a finite separable extension $K'/K$ such that 
there exists a strictly semi-stable scheme $\mathcal{Y}$ over $\sO_{K'}$ whose smooth locus is isomorphic to the smooth locus of the N\'{e}ron model of $X_{K'}$.
Then we can run a $K_{Y/\sO_{K'}}$-MMP by Theorem \ref{thm:semi-stable mmp}.
We will see that flips do not occur in this MMP.
Let $Y=Y_{0} \rightarrow \cdots \rightarrow Y_{r} \dashrightarrow Y_{r+1}$
be a part of the MMP, where $Y_{0} \rightarrow \cdots \rightarrow Y_{r}$ is the composition of divisorial contractions.
Let $C \subset Y_{r+1}$ be a flipped curve of $Y_{r} \dashrightarrow Y_{r+1}$.
Since the morphism $(Y_{r+1})_{K'} \rightarrow (Y_{r})_{K'}$ extends over the smooth locus $Y_{r+1,\mathrm{sm}}$,
the generic point of $y \in C$ does not contained in $Y_{r+1,\mathrm{sm}}$.
Since $(Y_{r+1}. Y_{r+1,s})$ is geometrically dlt, $y$ is contained in the intersection of two irreducible components of $Y_{r,s}$.
Therefore, there exists an exceptional prime divisor $E$ over $Y_{r+1}$ such that $a_{E}(Y_{r},Y_{r,s}) = 0$ and the center of $E$ on $Y$ is $C$.
Since $C$ is a flipped curve, 
we have $a_{E}(Y_{r},Y_{r,s}) < a_{E}(Y_{r+1},Y_{r+1,s})$.
This is the contradiction since $(Y_{r}, Y_{r,s})$ is geometrically dlt.
Let $Y'$ be the output variety of this MMP.
By the N\'{e}ron mapping property, we can show that
\[
\#\{\textup{irreducible components of } Y_{s}\} = \#\{\textup{irreducible components of }Y'_{s}\}.
\]
Therefore, we have $Y=Y'$ and $K_{Y/ \sO_{K'}}$ is nef.
By the argument in \cite[Lemma 4.7]{Maulik}, $K_{Y/\sO_{K'}}$ is trivial.
\end{proof}

The dual graph of the special fiber of a minimal strictly semi-stable model is classified in  \cite{chiarellotto-lazda} in the case where $\chara k \neq 2$. %Since they treat only the case where $\chara k \neq 2$, %we include the argument which works even in $\chara k = 2$.
We will verify that their result holds even in $\chara k = 2$.

\begin{thm}
\label{thm:combinatorial}
Let $\sO_{K}$, $K$, $k$, and $X$ be as in Definition \ref{defn: minimalstrictsemist}.
Let $\mathcal{X}$ be a minimal strictly semi-stable model over $\sO_{K}$.
Then the special fiber $\mathcal{X}_{k}$ is combinatorial in the sense of \cite[Definition 5.4, Definition 5.6]{chiarellotto-lazda}.
%\begin{enumerate}
%    \item 
%    If $X$ is a K3 surface, then one of the following holds.
%    \begin{enumerate}
%        \item 
%        The special fiber $X_{s}$ is a smooth K3 surface.
%        \item
%        The $X_{s} = Y_{1} \cup \ldots \cup Y_{N}$ is a chain with $Y_{1}, Y_{N}$ smooth rational 
%     \   
%    \end{enumerate}
%    \item
%    If $X$ is an abelian surface, then one of the following holds.
%    \begin{enumerate}
%        \item 
%        
%    \end{enumerate}
%\end{enumerate}
\end{thm}

\begin{proof}
If $X$ is a K3 surface, it follows from the same argument as in the proof of \cite[Proposition 5.3, Theorem 6.1]{chiarellotto-lazda}.
Therefore, we will treat the case where $X$ is an abelian surface.
We note that the case where $\chara k \neq 2$ follows from the proof of \cite[Proposition 5.3, Theorem 8.1]{chiarellotto-lazda}.
%By the same argument as in the proof of \cite[Proposition 5.3, Theorem 8.1]{chiarellotto-lazda},
%we have the $k$-betti numbers $\dim_{k} H^{i}_{\mathrm{sing}} (\Gamma, )$
In Chiarellotto and Lazda's argument, the assumption $\chara k \neq 2$ is used only in the case where (2) (b) in \cite[p.2253]{chiarellotto-lazda} holds for some irreducible component of the special fiber $\mathcal{X}_{k}$.
We will review their arguments in this case.
They show that the dual graph $\Gamma$ of the special fiber $\mathcal{X}_{k}$ is a triangulation of a compact real surface $M$ without border.
The spectral sequence of coherent cohomologies
shows that
\[
\dim_{k}
H^{i}_{\mathrm{sing}}(\Gamma, k) = 
\left\{
\begin{aligned}
1 \ \  &\textup{if } i =0,2, \\
2 \ \  &\textup{otherwise.} 
\end{aligned}
\right.
\]
It implies that $M$ is a torus if $\chara k \neq 2$ since  the left hand side is equal to $\C$-Betti number by the classification of real surfaces.
On the other hand, in the weight spectral sequence as in the proof of \cite[Theorem 8.3]{chiarellotto-lazda}, we have an isomorphism $E_{2}^{-1,2} \simeq E_{2}^{1,0}$. %is 2-dimensional since $H^{1}_{\'{e}t}(X_{\bar{K}},\Q_{\ell})$ is 2-dimensional.
%Since $H^{1}_{\'{e}t}(X_{\bar{K}},\Q_{\ell})$ is 2-dimensional, 
Here, we take a prime number $\ell \in R^{\times}$.
Moreover, we have $E_{1}^{0,1} = 0$ by \cite[Lemma 4.2]{chiarellotto-lazda}.
Since this spectral sequence degenerates at $E_{2}$, we have $\dim_{\Q_{\ell}} E_{2}^{1,0} = 2$.
Since $E_{2}^{1,0} = H^{1}_{\mathrm{sing}}(\Gamma, \Q_{\ell})$ as in the proof of \cite[Theorem 8.3]{chiarellotto-lazda}, we have $\dim_{\Q_{\ell}} H^{1}_{\mathrm{sing}}(\Gamma, \Q_{\ell}) = 2$.
Therefore, the surface $M$ is a torus even in $\chara k =2$.
\end{proof}
%In the case where $X$ is abelian surface, the minimal semi-stable model gives the compactification of the N\'{e}ron model.

In the case where $X$ is an abelian surface, %it is well-known that there exists a N\'{e}ron model of $X$ (cf.\, \cite{Bosch1990}), which is a smooth model satisfying a useful extension property (see \cite[Section 1.2, Definition 1]{Bosch1990} for the precise definition).
the strictly semi-stable scheme model which we obtained in the proof of Theorem \ref{thm: potentiallyminimalsemistable} gives a compactification of a N\'{e}ron model of $X$.
Conversely, the following proposition, which is proved in \cite[Theorem 1.4]{Jordan-Morrison} in the case where $\mathcal{X}$ is a scheme, shows that a minimal strictly semi-stable model gives a compactification of a N\'{e}ron model.
\begin{prop}
\label{prop:compactneron} 
Let $\sO_{K}$, $K$, and $k$ be as in Definition \ref{defn: minimalstrictsemist}.
Let $X$ be an abelian surface over $K$.
Let $\mathcal{X}$ be a minimal strictly semi-stable model of $X$ over $\sO_{K}$, and $\mathcal{X}^{\mathrm{sm}}$ the smooth locus of $\mathcal{X}$.
%Let $\mathcal{Y}$ be a N\'{e}ron model of $X$ over $\sO_{K}$.
%Then the natural map $f: \mathcal{X}^{\mathrm{sm}} \rightarrow  \mathcal{Y}$ is an isomorphism.
Then $\mathcal{X}^{\mathrm{sm}}$ is a scheme and a N\'{e}ron model of $X$ over $\sO_{K}$.
\end{prop}

\begin{proof}
Let $\mathcal{Y}$ be a N\'{e}ron model of $X$ over $\sO_{K}$.
By the descent argument, one can show that the N\'{e}ron mapping property holds for algebraic spaces.
Therefore, we have a morphism $f: \mathcal{X}^{\mathrm{sm}} \rightarrow \mathcal{Y}$ which extends an identity on $X$. It is enough to show that $f$ is an isomorphism.
We note that we have $\mathcal{X}(\sO_{K}^{\mathrm{sh}}) \neq \emptyset$ by Hensel's Lemma, where $\sO_{K}^{\mathrm{sh}}$ is the strict henselization of $\sO_{K}$. 
%By \cite[Section 6.5, 3]{Bosch1990},
By \cite[Section 7.2, Theorem 1]{Bosch1990}, we may assume that $X$ admits a section.
%Now the desired statement follos from \cite[]{Bosch1990}
First, we will show that $f$ is an \'{e}tale morphism.
%For any \'{e}tale morphism $u : U \rightarrow \mathcal{X}^{\mathrm{sm}}$, 
It suffices to show that $f \circ u$ is \'{e}tale for any \'{e}tale morphism $u : U \rightarrow \mathcal{X}^{\mathrm{sm}}$.
By \cite[Section 2.2, Corollary 10]{Bosch1990}, we want to show that
$\Phi:(f \circ u)^{\ast} \Omega^{2}_{\mathcal{Y}/ \Spec \sO_{K}} \rightarrow \Omega^{2}_{U/ \Spec \sO_{K}}$ is an isomorphism.
Since $\mathcal{X}^{\mathrm{sm}}$ is a scheme in codimension $1$, by using \cite[Section 4.3, Lemma 1]{Bosch1990}, we have $\Phi$ is an isomorphism in codimension $1$, so $\Phi$ is an isomorphism. Now we have $f$ is \'{e}tale.
By Zariski's main theorem, the morphism $f$ is an open immersion.
Let $Z$ be a complement $\mathcal{Y} \setminus \mathcal{X}^{\mathrm{sm}}$.
Suppose that $Z\neq \emptyset$.
Take a valued point $z \in Z(\overline{k})$.
By Hensel's lemma, the valued point $z$ lifts to $\widetilde{z} \in \mathcal{Y} (\sO_{K}^{\mathrm{sh}}) = X (K^{\mathrm{sh}}) = \mathcal{X} (K^{\mathrm{sh}}) = \mathcal{X} (\sO_{K}^{\mathrm{sh}})$, but it contradicts to the choice of $z$ (cf.\, the proof of \cite[Theorem 1.4]{Jordan-Morrison}).
Therefore, $f$ is an isomorphism.
\end{proof}

\section{More general relative MMP}
The goal of this section is to prove the relative MMP in more general setting (Theorem \ref{thm:relative mmp over dvr'}) and the finite generation of relative canonical ring (Theorem \ref{prop:abundance over perfect fields}).
The first one is an analog of \cite[Theorem 1.6]{hacon-witaszek19} in mixed characteristic.
In this section, \cite[Theorem 1.8]{witaszek20} plays an essential role to prove such theorems.
We note that \cite[Theorem 1.8]{witaszek20} will be proved in the upcoming paper \cite{witaszek21}.

\begin{prop}\label{prop:contraction trivial case}
Let $V$ be the spectrum of an excellent discrete valuation ring.
Let $X$ be a flat $V$-variety of relative dimension two.
Let $(X,\Delta)$ be a $\Q$-factorial dlt log pair with $X_s \subset \lfloor \Delta \rfloor$ as sets, where $X$ is flat $V$-variety of relative dimension two.
Let $\rho \colon X \to U$ be a projective morphism to a quasi-projective $V$-variety $U$.
Let $L$ be a $\rho$-nef Cartier divisor on $X$ with $L^{\perp}=\R[\Sigma]$, where 
 $\Sigma$ is a $(K_X+\Delta)$-negative extremal ray over $\rho$.
Then $L$ is semiample.
\end{prop}

\begin{proof}
Replacing $L$ with $mL$ for large enough $m$, we may assume that $L-(K_{X/V}+\Delta)$ is ample over $U$.
Let $S_1, \ldots, S_r$ be the irreducible components of $X_s$.
We denote the different $\mathrm{Diff}_{S_i^N}(\Delta-S_i)$ by $D_i$, then $(S^N_i,D_i)$ is dlt and $L-(K_{S_i^N}+D_i)$ is ample over $U$.
We note that the normalization $S_i^N \to S_i$ is a universal homeomorphism by Lemma \ref{lem:topologically normal}.
By \cite[Theorem 1.1]{tanaka20-imperfect} and the proof of \cite[Lemma 1.4]{keel}, $L|_{S_i}$ is semiample for all $i$.
We denote the map induced by $L|_{S_i}$ by $\phi_i \colon S_i \to T_i$.
Since $-(K_{S_i^N}+D_i)$ is ample over $T_i$, the morphism $\phi_i|_{S_1 \cup \cdots \cup S_{i-1}}$ has connected fibers by \cite[Theorem 5.2]{tanaka18}.
By \cite[Corollary 2.9]{keel} and the induction on $i$,  $L|_{X_s}$ is semiample.
Furthermore, $L|_{X_\eta}$ is semiample by the base point free theorem.
Thus, by \cite[Theorem 1.8]{witaszek20}, $L$ is also semiample.
\end{proof}

\begin{thm}\label{thm:relative mmp over dvr'}\textup{(cf.\,\cite[Theorem 1.6]{hacon-witaszek19})}
Let $V$ be the spectrum of an excellent discrete valuation ring.
Let $(X,\Delta)$ be a $\Q$-factorial dlt pair over $V$, where $X$ is a flat $V$-variety of relative dimension two.
Let $X \to Z$ be a projective morphism over $V$ to a quasi-projective $V$-variety $Z$.
Assume that $\lfloor \Delta \rfloor$ contains the closed fiber $X_s$ as sets.
%Assume that \cite[Theorem 1.8]{witaszek20} holds.
Then we can run a $(K_{X/V}+\Delta)$-MMP over $Z$ which terminates with a minimal model or a Mori fiber space.
\end{thm}

\begin{proof}
By the cone theorem (Proposition \ref{prop:special fiber cone thm}) and the contraction theorem (Proposition \ref{prop:contraction trivial case}), we can contract any $(K_X+\Delta)$-negative extremal ray.
If the contraction is a flipping contraction, then the flip exists by the argument in the proof of Theorem \ref{thm:semi-stable mmp'}.
The termination of flips follows from the special termination (Proposition \ref{prop:special termination}).
\end{proof}

\begin{thm}\label{prop:abundance over perfect fields}
Let $V$ be an excellent Dedekind scheme whose each residue field is perfect.
Let $X$ be a flat projective $V$-variety of relative dimension two.
Let $(X,\Delta)$ be a dlt pair over $V$ such that for each closed point $s \in V$, the pair $(X,\Delta+X_s)$ is dlt or $X_s$ is contained in $\lfloor \Delta \rfloor$ as sets.
%Let $\rho \colon X \to U$ be a projective morphism to a projective $V$-variety.
Then 
\[
R(K_{X/V}+\Delta) := \bigoplus_m H^0(X,\sO_X(m(K_{X/V}+\Delta)))
\]
is finitely generated.
\end{thm}

\begin{proof}
We may assume that $V$ is the spectrum of an excellent discrete valuation ring and $\lfloor \Delta \rfloor$ contains the closed fiber $X_s$ as sets.
%We may assume that $U$ is affine.
%By Lemma \ref{lem:q-factorial compactification},
%we may assume that $X$ and $U$ are projective over $V$ and $X$ is $\Q$-factorial.
If $K_{X/V}+\Delta$ is not pseudo-effective, the theorem is trivial, thus, we may assume that $K_{X/V}+\Delta$ is pseudo-effective.
By Theorem \ref{thm:relative mmp over dvr'}, we can run a $(K_{X/V}+\Delta)$-MMP over $V$, thus, we may assume that $L:=K_{X/V}+\Delta$ is nef over $V$.
By \cite[Theorem 1.8]{witaszek20}, it is enough to show that $L|_{X_s}$ is semiample over $V$.
%We take an ample divisor $H$ on $U$ such that $\rho^*H+L$ is nef over $V$.
By the argument of the proof of \cite[Theorem 1.6]{hacon-witaszek19}, the pair $(W,\Delta_W)$ is a sdlt surface and $\pi$ is a universal homeomorphism, where $\pi \colon W \to X_s$ is the $S_2$-fication and $L_W:=K_W+\Delta_W=\pi^*((K_X+\Delta)|_{X_{s}})$, thus it is enough to show that $L_W$ is semiample over $U$.
%We set $H_W:=\pi^*\rho^*H$.
%Then $L_W+H_W$ is nef and it is enough to show that $L_W+H_W$ is semiample.
%By the proof of \cite[Theorem 1]{tanaka17}, there exists an effective $\Q$-Weil divisor $D$ which is $\Q$-linearly equivalent to $H_W$ such that $(W,\Delta_W+D)$ is also sdlt.
By \cite[Theorem 0.1]{tanaka16}, it is semiample.
\end{proof}

\begin{rmk}\label{rmk:relative remark}
%If $X$ is projective over $U$, running a minimal model program for $(X,\Delta)$, we obtain a
Theorem \ref{prop:abundance over perfect fields} holds in the relative setting if $p\neq 2,3,5$. 
More precisely, if $(X,\Delta), V$ are the same as in Theorem \ref{prop:abundance over perfect fields} and $\rho \colon X \to U$ is a projective morphism to a projective $V$-variety $U$, then
\[
R(K_{X/V}+\Delta/U) := \bigoplus_m \rho_*\sO_X(m(K_{X/V}+\Delta))
\]
is finitely generated $\sO_U$-algebra.
Indeed, by the same argument,we may assume that $L:=K_{X/V}+\Delta$ is nef over $U$.
By \cite[Theorem H]{bmpstww}, there exists an ample Cartier divisor $H$ on $U$ such that $L+\rho^*H$ is nef over $V$.
By the  proof of Theorem \ref{prop:abundance over perfect fields} and  \cite[Theorem 1]{tanaka17},  we obtain the semiampleness of $L+\rho^*H$.
\end{rmk}

\begin{comment}
\begin{rmk}
In our proof, we need the projectivity of $V$ to use \cite[Theorem 0.1]{tanaka16}.
If we assume the existence of projective log resolutions, then we generalize the existence of dlt modifications to the general setting.
Then, we can take a $\Q$-factorial dlt compactification for a pair $(X,\Delta)$ with $X_s \subset \lfloor \Delta \rfloor$ as sets and generalize Theorem \ref{prop:abundance over perfect fields} to the quasi-projective case. 
\end{rmk}
\end{comment}

\begin{cor}\label{thm:abundance}
Let $V$ be an excellent Dedekind scheme.
Let $X$ be a projective $V$-variety satisfying Assumption \ref{assump:semi-stable}.
%Let $\rho \colon X \to U$ be a projective morphism to a projective $V$-variety $U$.
Then 
\[
R(K_{X/V}) := \bigoplus_m H^0(X,\sO_X(mK_{X/V}))
\]
is finitely generated.
\end{cor}

\begin{proof}
By Assumption \ref{assump:semi-stable} $(6)$, we may assume that $V$ is the spectrum of a discrete valuation ring with perfect residue field.
Thus, Theorem \ref{thm:abundance} follows from Theorem \ref{prop:abundance over perfect fields}.
\end{proof}

%\bibliography{bibliography}
%\bibliographystyle{amsalpha}
\printbibliography
\end{document}